\documentclass[11pt,letterpaper]{amsart}
\usepackage{amsthm,amssymb,amsfonts}
\usepackage[nosumlimits]{mathtools}
\usepackage{mathrsfs}
\usepackage{graphicx}
\usepackage{accents}
\usepackage{algorithm}
\usepackage{algpseudocode}
\usepackage{tikz-cd}
\usepackage{enumerate}
\usepackage{quiver}
\usepackage{stmaryrd}
\usepackage{enumitem}
\usepackage{subcaption}

\usepackage[margin=1in]{geometry}

\usepackage{hyperref,xcolor}
\hypersetup{
    colorlinks,
    linkcolor={red!50!black},
    citecolor={blue!50!black},
    urlcolor={blue!80!black}
}
\usepackage{adjustbox}

\newcommand{\tp}{{\scriptscriptstyle\mathsf{T}}}
\newcommand{\s}{{\scriptscriptstyle\operatorname{S}}}
\newcommand{\p}{{\scriptscriptstyle\operatorname{P}}}
\newcommand{\w}{{\scriptscriptstyle\Lambda}}

\let\oldlhd\unlhd
\renewcommand*{\unlhd}{\mathrel{\mskip0.1mu \oldlhd \mskip0.1mu}}

\newcommand{\smallprod}{{\textstyle\prod}}
\newcommand{\medotimes}{{\textstyle\bigotimes}}

\theoremstyle{plain}
\newtheorem{theorem}{Theorem}[section]
\newtheorem{proposition}[theorem]{Proposition}
\newtheorem{lemma}[theorem]{Lemma}
\newtheorem{corollary}[theorem]{Corollary}
\newtheorem{remark}[theorem]{Remark}
\theoremstyle{definition}
\newtheorem{definition}[theorem]{Definition}

\DeclareMathOperator{\GL}{GL}

\DeclareMathOperator{\im}{Im}

\DeclareMathOperator{\pr}{PR}

\DeclareMathOperator{\Ker}{Ker}

\DeclareMathOperator{\AR}{AR}
\DeclareMathOperator{\GR}{GR}

\DeclareMathOperator{\Q}{Q}
\DeclareMathOperator{\codim}{codim}
\DeclareMathOperator{\spa}{span}
\DeclareMathOperator{\Hom}{Hom}

\DeclareMathOperator{\PR}{PR}

\DeclareMathOperator{\Gr}{Gr}
\DeclareMathOperator{\height}{ht}
\DeclareMathOperator{\Ass}{Ass}
\DeclareMathOperator{\rank}{rank}

\DeclareMathOperator{\ch}{char}

\DeclareMathOperator{\Alt}{Alt}
\DeclareMathOperator{\Sym}{Sym}
\DeclareMathOperator{\Frac}{Frac}

\DeclareMathOperator{\Irr}{Irr}
\DeclareMathOperator{\cp-rank}{cp-rank}
\DeclareMathOperator{\poly}{poly}
\DeclareMathOperator{\PSPACE}{PSPACE}
\DeclareMathOperator{\AM}{AM}

\begin{document}
\title{Isotropy and Completeness Indices of Multilinear Maps}
\date{}
\author[Q-Y.~Chen]{Qiyuan~Chen}
\address{State Key Laboratory of Mathematical Sciences, Academy of Mathematics and Systems Science, Chinese Academy of Sciences, Beijing 100190, China}
\email{chenqiyuan@amss.ac.cn}
\author[K.~Ye]{Ke Ye}
\address{State Key Laboratory of Mathematical Sciences, Academy of Mathematics and Systems Science, Chinese Academy of Sciences, Beijing 100190, China}
\email{keyk@amss.ac.cn}
\begin{abstract}
Structures of multilinear maps are characterized by invariants. In this paper we introduce two invariants, named the isotropy index and the completeness index. These invariants capture the tensorial structure of the kernel of a multilinear map.  We establish bounds on both indices in terms of the partition rank, geometric rank, analytic rank and height,  and present three applications: 1) Using the completeness index as an interpolator, we establish upper bounds on the aforementioned tensor ranks in terms of the subrank. This settles an open problem raised by Kopparty, Moshkovitz and Zuiddam, and consequently answers a question of Derksen, Makam and Zuiddam.  2) We prove a Ramsey-type theorem for the two indices, generalizing a recent result of Qiao and confirming a conjecture of his.  3) By computing the completeness index, we obtain a polynomial-time probabilistic algorithm to estimate the height of a polynomial ideal.
\end{abstract}
\maketitle

\section{Introduction}
Multilinear maps are natural generalizations of linear maps and play a prominent role across diverse branches of mathematics by virtue of a variety of rich structures.  Various invariants have been introduced to characterize these structures from different perspectives.  For instance,  the \emph{CP-rank} was first proposed in~\cite{Hitchcock27} to analyze multi-dimensional data; it was later exploited by Strassen \cite{Strassen69},  who also defined the \emph{subrank}~\cite{Strassen87},  to measure the complexity of bilinear maps.  Different versions of the \emph{hyperdeterminant} were given \cite{GKZ92,GKZ94,Cayley09} to generalize the determinant.  The \emph{Schmidt rank}~\cite{Schmidt85} was introduced to study rational points of algebraic varieties,  and was recast as \emph{strength} in \cite{ananyan2020small} to resolve the Stillman conjecture in commutative algebra.  The notion of \emph{slice rank},  formulated by Tao~\cite{Terrence16} and generalized to the \emph{partition rank} \cite{Naslund20},  was implicitly used in \cite{CLP17,EG17} to study the cap set problem in combinatorics.  To measure the randomness of multilinear functions over finite fields,  Gowers and Wolf developed the \emph{analytic rank} in \cite{GW11}.  The \emph{geometric rank},  an analogue of the analytic rank over arbitrary fields,  was proposed in \cite{kopparty2020geometric}.  The \emph{isotropic index},  emerging from the study of groups~\cite{Atkinson73,buhler1987isotropic},  has recently become a central topic at the interface between combinatorics and computer science~\cite{Dav1992number, BMW17, BCGQS21,qiao2020tur}.  These invariants,  among others~\cite{AH95,BB21,Lim05},  form a mosaic of multilinear maps.

Yet each invariant illuminates only one facet of multilinear maps---emerging evidence reveal that these seemingly disparate notions are intimately intertwined,  indicating that \emph{different structures of multilinear maps are inherently related}.  By way of illustration,  we recall Gowers and Wolf conjectured in~\cite{GW11} that the partition rank is equivalent to the analytic rank,  accordingly suggesting the equivalence between the decomposition and randomness structures of multilinear maps.  This conjecture has been almost resolved~\cite{cohen2023partition,moshkovitz2022quasi}.  A relation between Cayley's hyperdeterminant and the partition rank was recently discovered~\cite{AY23},  implying an interplay between the combinatorial and decomposition structures. In~\cite{kopparty2020geometric}, the question of whether the geometric rank relates to the subrank was left open.  A partial answer in~\cite{chen2025bounds} demonstrates a link between the geometric and computational structures of multilinear maps.  Other typical examples include the equivalence among the slice rank,  geometric rank and analytic rank for bilinear maps~\cite{cohen2021structure},  the equivalence between the geometric rank and analytic rank for arbitrary multilinear maps~\cite{chen2024stability},  and the equality of the strength and slice rank for generic multilinear maps~\cite{BBOV23}.  Along these lines,  we introduce in this work the isotropy and completeness indices for arbitrary (resp.  alternating,  symmetric) multilinear maps,  and establish relations between these indices and other invariants.  Consequently,  we obtain several interesting results,  each of which either resolves a conjecture,  or strengthens an existing result.  Our main contributions are divided into two parts,  summarized below in Subsections~\ref{subsec:main thm} and \ref{subsec:app}.

\subsection{Main theorems}\label{subsec:main thm}
The isotropy index $\alpha(F)$ and the completeness index $\beta(F)$ of a $\mathsf{K}$-multilinear map $F: \mathbb{V}_1 \times \cdots \times  \mathbb{V}_d \to \mathbb{W}$ quantify the shape of $\Ker(F) \coloneqq \lbrace 
T \in \mathbb{V}_1 \otimes \cdots \otimes \mathbb{V}_{d}: F(T) = 0
\rbrace$,  where we view $F$ as a linear map from $\mathbb{V}_1 \otimes \cdots \otimes  \mathbb{V}_d$ to $\mathbb{W}$.  When $F$ is alternating (resp.  symmetric),  the isotropy and completeness indices are denoted by $\alpha_{\w}(F)$ and $\beta_{\w}(F)$ (resp.  $\alpha_{\s}(F)$ and $\beta_{\s}(F)$).  The precise definitions of these indices can be found in Definition~\ref{def:index}.  Let $T_F$ be the tensor in $\mathbb{V}_1^\ast \otimes \cdots \otimes  \mathbb{V}_d^\ast \otimes \mathbb{W}$ corresponding to $F$ via the natural isomorphism $\Hom(\mathbb{V}_1 \times \cdots \times \mathbb{V}_d,  \mathbb{W}) \simeq \mathbb{V}_1^\ast \otimes \cdots \otimes  \mathbb{V}_d^\ast \otimes \mathbb{W}$.  We write $\PR(T),  \GR(T),  \AR(T)$ for the partition rank,  geometric rank,  analytic rank of a tensor $T$,  respectively.   We refer the reader to Definition~\ref{def:rank} for the definitions of these tensor ranks.  The main results of this paper are Theorems~\ref{lem-isotropic-general} and \ref{lem-complete1},  in which we respectively bound the isotropy and completeness indices of $F$ by various ranks of $T_F$.  The proofs of the two theorems are deferred to Section~\ref{sec:estimate}.
\begin{theorem}[Bounds on isotropy index]\label{lem-isotropic-general} 
Let $\mathsf{K}$ be a field and let $\mathbb{V}_1,\dots,  \mathbb{V}_d,  \mathbb{W}$ be vector spaces over $\mathsf{K}$.  Denote $m  \coloneqq \min\{\dim \mathbb{V}_i: i \in [d]\}$.  Suppose $F: \mathbb{V}_1 \times \cdots \times \mathbb{V}_d \to \mathbb{W}$ is a nonzero multilinear map.  Then 
\begin{enumerate}[label=(\alph*)]
\item\label{lem-isotropic-general:item1} $m \le \PR(T_F) (\alpha(F)+1)^{d-1} + \alpha(F) + 1$.
\item\label{lem-isotropic-general:item2} $m \le \PR(T_F) \binom{\alpha_{\w}(F)}{d-1}+ \alpha_{\w}(F)$,  if $\mathbb{V}_1 = \cdots = \mathbb{V}_d$ and $F$ is an alternating multilinear map.
\item\label{lem-isotropic-general:item3} $m \le \PR(T_F) \binom{\alpha_{\s}(F) + d - 1}{d-1} + \alpha_{\s}(F)$,  if $\mathsf{K}$ is algebraically closed,  $\mathbb{V}_1 = \cdots = \mathbb{V}_d$ and $F$ is a symmetric multilinear map.
\end{enumerate}
\end{theorem}

By Theorem~\ref{lem-isotropic-general},  a multilinear map with small isotropy index must have large partition rank.  However,  Theorem~\ref{lem-complete1} below shows the situation for the completeness index is reversed.  That is,  a multilinear map with small isotropy index must have small geometric and analytic ranks. 
\begin{theorem}[Bounds on completeness index] \label{lem-complete1} 
Let $\mathbb{V}_1,\dots,  \mathbb{V}_d,  \mathbb{W}$ be vector spaces over a field $\mathsf{K}$.  Suppose $F: \mathbb{V}_1 \times \cdots \times \mathbb{V}_d \to \mathbb{W}$ is a nonzero multilinear map.  
\begin{enumerate}[label = (\alph*)]
\item\label{lem-complete1:gen} We have $\GR(T_F) < (\beta (F)+1)^{d}$ for infinite $\mathsf{K}$,  and $\AR(T_F) < d+(\beta (F)+1)^{d}+ \left\lceil\frac{(\beta (F)+1)^{d}(d-1)}{|\mathsf{K}|-1} \right\rceil$ for finite $\mathsf{K}$.
\item\label{lem-complete1:alt} If $\mathbb{V}_1 = \cdots = \mathbb{V}_d$ and $F$ is an alternating multilinear map,  then $\GR(T_F) < \binom{\beta_{\w}(F)+1}{d}$ for infinite $\mathsf{K}$,  whereas 
$\AR(T_F) < d+\binom{\beta_{\w}(F)+1}{d}+\left\lceil\frac{\binom{\beta_{\w}(F)+1}{d}(d-1)}{|\mathsf{K}| -1} \right\rceil$ for finite $\mathsf{K}$.
\item\label{lem-complete1:sym} Assume either $\ch(\mathsf{K})=0$ or $\ch(\mathsf{K}) > d$ and $|\mathsf{K}| >  \height(\mathfrak{a}_F)$.  If $\mathbb{V}_1 = \cdots = \mathbb{V}_d = \mathbb{V}$ and $F$ is a symmetric multilinear map ,  then $\beta_{\s}(F)\le \height (\mathfrak{a}_F) < \binom{\beta_{\s}(F)+d}{d}$.  Moreover,  we have $\beta_{\s}(F)\le  \GR (T_F) < \binom{\beta_{\s}(F)+d}{d}$ if $\mathsf{K}$ is infinite.  Here $\mathfrak{a}_F \subseteq \mathbb{K}[x_1,\dots,  x_m]$ is the ideal generated by degree $d$ homogeneous polynomials obtained from $F$ by choosing bases of $\mathbb{V}$ and $\mathbb{W}$,  and $m = \dim \mathbb{V}$ and $n = \dim \mathbb{W}$.
\end{enumerate}
\end{theorem}
\subsection{Applications}\label{subsec:app}
Theorems~\ref{lem-isotropic-general} and \ref{lem-complete1} yield a number of interesting consequences in Section~\ref{sec:applications}.  Below we summarize the most important ones.  
\subsubsection{Geometric rank and subrank}
In \cite[Section~9]{kopparty2020geometric},  an open problem concerning the relation between the geometric rank and subrank was raised.  It was shown in \cite{chen2025bounds} that the geometric rank is upper bounded by a function $f$ of the subrank.  Moreover,  it was proved in the same paper that $f$ is a quadratic polynomial for order-$3$ tensors.  This provides evidence for the Geometric Rank vs.  Subrank Conjecture \cite[Conjecture~8.1]{chen2025bounds},  stating that $f$ is a polynomial of degree $(k-1)$ for order-$k$ tensors.  In the following result to be proved in Subsection~\ref{subsec:rank},  we not only resolve this conjecture,  but also prove that the degree $(k-1)$ is optimal.  Consequently,  this fully settles the open problem in \cite{kopparty2020geometric}.  In the following,  $\Q(T)$ denotes the subrank of a tensor $T$ (see Definition~\ref{def:rank}).
\begin{theorem}[Subrank vs.  geometric rank]\label{cor:GvsS}
Suppose $\mathsf{K}$ is a field and $\mathbb{V}_1,\dots,  \mathbb{V}_k$ are vector spaces over $\mathsf{K}$.  For any $T \in \mathbb{V}_1 \otimes \cdots \otimes \mathbb{V}_k$,  we have $\Q(T) \le \GR(T) \lesssim_k \Q(T)^{k-1}$.  Moreover,  if $\dim \mathbb{V}_j \ge 2k + 2$ for each $j \in [k]$,  there exists some $T_0 \in \mathbb{V}_1 \otimes \cdots \otimes \mathbb{V}_k$ such that $\GR(T_0) \asymp_k \Q(T_0)^{k-1} $.
\end{theorem}
Since the geometric rank is additive with respect to the direct sum,  Theorem~\ref{cor:GvsS} immediately leads to an estimate for the subrank of the direct sum. 
\begin{corollary}[Subrank of direct sum]
\label{cor:add}
Suppose $\mathsf{K}$ is a field and $S, T$ are two tensors over $\mathsf{K}$ of order $k$,  then we have $\Q(S) + \Q(T) \le \Q(S \oplus T) \lesssim_k  \Q(S)^{k-1} + \Q(T)^{k-1}$.
\end{corollary}
For any positive integers $n$ and $k$,  \cite[Theorem~5.2]{derksen2024subrank} constructs tensors $S_0, T_0 \in (\mathsf{K}^{n})^{\otimes k}$ such that $\Q(S_0 \oplus T_0) \gtrsim_k \Q(S_0)^{k-1} + \Q(T_0)^{k-1}$.  This implies that the upper bound in Corollary~\ref{cor:add} is optimal,  up to a constant factor.  Thus,  Corollary~\ref{cor:add} resolves the open problem posed in Section~6 of \cite{derksen2024subrank},  which asked for a relation between $\Q(S \oplus T)$ and $\Q(S),  \Q(T)$.  For $k = 3$,  Corollary~\ref{cor:add} can be proved by a completely different method \cite[Corollary~1.8]{chen2025bounds}.

In Subsection~\ref{subsec:rank},  we will also obtain Corollaries~\ref{cor:svsp}--\ref{cor:debor},  which generalize existing results of \cite{chen2025bounds} and \cite{biaggi2025real} from order-$3$ tensors to order-$k$ tensors.

\subsubsection{Ramsey numbers}
In Definition~\ref{def:Ramsey},  we define the Ramsey number $R(\mathsf{K},d,s,t)$ to be the minimum integer $m$ such that any $\mathsf{K}$-multilinear map $F: \mathbb{V}_1 \times \cdots \times \mathbb{V}_d \to \mathbb{W}$ satisfies either $\alpha(F) \ge s$ or $\beta(F) \ge t$ whenever $\min_{j \in [d]} \dim \mathbb{V}_j \ge m$.  For the alternating (resp.  symmetric) case,  we similarly define the Ramsey number $R_{\w}(\mathsf{K},d,s,t)$ (resp.  $R_{\s}(\mathsf{K},d,s,t)$).  A priori,  the Ramsey numbers need not even exist; however,  using Theorems~\ref{lem-isotropic-general} and~\ref{lem-complete1},  we will establish in Subsection~\ref{subsec:Ramsey} their existence and upper bounds.
\begin{theorem}[Existence and upper bounds for Ramsey numbers]\label{prop:Ramsey}
Given nonnegative integers $d,  s,t$,  the Ramsey numbers $R(\mathsf{K},d,s,t)$ and $R_{\w}(\mathsf{K},d,s,t)$ exist for any field $\mathsf{K}$,  while $R_{\s}(\mathsf{K},d,s,t)$ exists when $\mathsf{K}$ is algebraically closed.   Moreover,  in each case,  we have
\begin{enumerate}[label = (\alph*)]
\item\label{prop:Ramsey:item1}$R(\mathsf{K},d,s,t) \lesssim_d s^{d-1}t^{d^2}$ if $|\mathsf{K}| = \infty$,  and $R(\mathsf{K},d,s,t) \lesssim_d s^{d-1}t^{d}(\log_q s + \log_q t) $ if $|\mathsf{K}| = q$.
\item\label{prop:Ramsey:item2} $R_{\w}(\mathsf{K},d,s,t) \lesssim_d s^{d-1}t^{d^2}$ if $|\mathsf{K}| = \infty$,  and $R_{\w}(\mathsf{K},d,s,t) \lesssim_d s^{d-1}t^{d}(\log_q s + \log_q t) $ if $|\mathsf{K}| = q$.
\item\label{prop:Ramsey:item3} $R_{\s}(\mathsf{K},d,s,t) \lesssim_d s^{d-1}t^{d}$.
\end{enumerate}
\end{theorem}
The existence of $R_{\w}(\mathsf{K},d,s,t)$ in Theorem~\ref{prop:Ramsey} extends \cite[Theorem~4.1]{qiao2020tur} from bilinear to multilinear maps,  and completely resolves the conjecture \cite[Conjecture~5.5]{qiao2020tur}.  We will also derive a lower bound for $R_{\w}(\mathsf{K},  d,  s,t)$. 
\begin{proposition}[Lower bound for $R_{\Lambda}(\mathsf{K},d,s,t)$]\label{ramesey lower bound}
For any field $\mathsf{K}$ and nonnegative integers $d,  s$ and $t$,  we have 
\[
R_{\w}(\mathsf{K},  d,  s,t) \ge  \frac{\binom{s-1}{d}}{s-1} \left[ \binom{t}{d}-1 \right] + s \asymp_d s^{d-1} t^d.
\]
In particular,  over a field $\mathsf{K}$ satisfying $\PR(T) \lesssim_k \GR(T)$ for any order-$k$ tensor $T$ over $\mathsf{K}$,  we have $R_{\w}(\mathsf{K},  d,  s,t) \asymp_d s^{d-1} t^d$. 
\end{proposition}
For $d = 2$,  the lower bound in Proposition~\ref{ramesey lower bound} has been essentially established in \cite[Proposition~5.2]{qiao2020tur}.  If moreover $\mathsf{K}$ is perfect, then Proposition~\ref{ramesey lower bound} shows that $R_{\w}(\mathsf{K},2, s,t) \approx_k st^2$.  This closes the gap between the upper and lower bounds obtained in \cite{qiao2020tur},  and thereby answers the open problem raised in \cite[Subsection~5.5]{qiao2020tur}.   

In Subsection~\ref{subsec:Ramsey},  we discuss applications of Theorem~\ref{prop:Ramsey} to multilinear algebra,  group theory and algebraic geometry,  and obtain Corollaries~\ref{cor:matrixspace}, \ref{cor:Ramseygroup}--\ref{cor:var},  which strengthen the existing results \cite[Corollary~3]{Dav1992number} and \cite[Corollaries~4.2 and 4.4]{qiao2020tur}.

\subsection{Organization}
In Section~\ref{sec:prelim},  we provide a brief overview of multilinear maps,  tensor ranks and algebraic geometry over finite fields.  We define the isotropy and completeness indices in Section~\ref{sec:indices},  discuss their basic properties and present some concrete examples.  We prepare two technical results in Section~\ref{sec:alg} that are necessary for the proofs of Theorems~\ref{lem-isotropic-general} and \ref{lem-complete1} given in Section~\ref{sec:estimate}.  Section~\ref{sec:applications} consists of three applications of Theorems~\ref{lem-isotropic-general} and \ref{lem-complete1}: the comparison of tensor ranks,  the bound for Ramsey numbers and the probabilistic algorithm for the completeness index.  

\section{Preliminaries}\label{sec:prelim}
In this section,  we fix basic notations and recall some facts that will be used in the subsequent sections.  For each positive integer $n$,  we denote $[n] \coloneqq \{1,\dots,  n\}$.  We use $\mathsf{F},  \mathsf{K}$ to denote general fields,  and we reserve $\mathsf{F}_q$ for the finite field of $q$ elements.  Let $\mathbb{V}_1,\dots,  \mathbb{V}_d,  \mathbb{W}$ be vector spaces over a field  $\mathsf{K}$.  We denote
\begin{align*}
\Hom(\mathbb{V}_1 \times \cdots \times \mathbb{V}_{d},  \mathbb{W}) &\coloneqq \left\lbrace
\text{$\mathsf{K}$-multilinear maps from~}\mathbb{V}_1 \times \cdots \times \mathbb{V}_{d} \text{~to~} \mathbb{W}
\right\rbrace,  \\
\Hom(\mathbb{V}_1 \otimes \cdots \otimes \mathbb{V}_{d},  \mathbb{W}) &\coloneqq \left\lbrace
\text{$\mathsf{K}$-linear maps from~}\mathbb{V}_1 \otimes \cdots \otimes \mathbb{V}_{d} \text{~to~} \mathbb{W}
\right\rbrace. 
\end{align*}
By definition,  we have 
\begin{equation}\label{eq:isom}
\Hom(\mathbb{V}_1 \times \cdots \times \mathbb{V}_{d},  \mathbb{W}) \simeq \Hom(\mathbb{V}_1 \otimes \cdots \otimes \mathbb{V}_{d},  \mathbb{W}) \simeq \mathbb{V}_1^\ast \otimes \cdots \otimes \mathbb{V}_{d}^\ast \otimes \mathbb{W}
\end{equation}
In particular,  if $\mathbb{V}_1 = \cdots = \mathbb{V}_{d} = \mathbb{V}$,  then \eqref{eq:isom} induces 
\begin{align*}
\Alt(\mathbb{V}^{d},\mathbb{W}) &\simeq \Hom( \Lambda^{d} \mathbb{V},  \mathbb{W} ) \simeq \Lambda^{d} \mathbb{V}^\ast \otimes \mathbb{W},  \\ 
\Sym(\mathbb{V}^{d},\mathbb{W}) &\simeq \Hom( \mathrm{S}^{d} \mathbb{V},  \mathbb{W} ) \simeq \mathsf{S}^{d} \mathbb{V}^\ast \otimes \mathbb{W}.
\end{align*}
Here $\Lambda^{d} \mathbb{V}$ (resp.  $\mathrm{S}^{d} \mathbb{V}$) is the subspace of $\mathbb{V}^{\otimes d}$ consisting of alternating (resp.  symmetric) tensors,  and 
\begin{align*}
\Alt^{d}(\mathbb{V},  \mathbb{W}) &\coloneqq \left\lbrace
\text{alternating $\mathsf{K}$-multilinear maps from~} \mathbb{V}^{d} \text{~to~} \mathbb{W}
\right\rbrace,  \\
\Hom( \Lambda^{d} \mathbb{V},  \mathbb{W} ) &\coloneqq \left\lbrace
\text{$\mathsf{K}$-linear maps from~} \Lambda^{d} \mathbb{V} \text{~to~} \mathbb{W}
\right\rbrace,   \\
\Sym^{d}(\mathbb{V},  \mathbb{W}) &\coloneqq \left\lbrace
\text{symmetric $\mathsf{K}$-multilinear maps from~} \mathbb{V}^{d} \text{~to~} \mathbb{W}
\right\rbrace,  \\
\Hom( \mathrm{S}^{d}\mathbb{V},  \mathbb{W}) &\coloneqq \left\lbrace
\text{$\mathsf{K}$-linear maps from~} \mathrm{S}^{d} \mathbb{V} \text{~to~} \mathbb{W}
\right\rbrace.
\end{align*}
In the sequel,  we will view these isomorphic vector spaces as indistinguishable and will identify their corresponding elements.

\subsection{Tensor ranks}\label{subsec:prelim-rank}
Let $\mathbb{W}_1,\dots,  \mathbb{W}_k,  \mathbb{W}'_1,\dots,  \mathbb{W}'_k$ be vector spaces over $\mathsf{K}$.  Suppose $T \in \mathbb{W}_1 \otimes \cdots \otimes \mathbb{W}_k$ and $S \in  \mathbb{W}'_1 \otimes \cdots \otimes \mathbb{W}'_k$ are two tensors.  We write 
\[
T = \sum_{j=1}^r w_{1,j}\otimes \cdots \otimes w_{k,j},\quad w_{i,j} \in \mathbb{W}_i,\; (i,j) \in [k] \times [r].
\]
We say that $S$ is a \emph{restriction} of $T$,  denoted as $S \unlhd T$,  if there exists some $(g_1,\dots,  g_k) \in \prod_{i=1}^k \Hom(\mathbb{W}_i, \mathbb{W}'_i)$ such that $S = (g_1,  \dots,  g_k) \cdot T$.  Here 
\begin{equation}\label{eq:action}
(g_1,  \dots,  g_k) \cdot T \coloneqq \sum_{j=1}^r g_1 (w_{1,j}) \otimes \cdots g_k(w_{k,j}).
\end{equation}
Given distinct integers $i_1,\dots,  i_s \in [k]$ and $\ell_{i_t} \in \mathbb{W}_{i_t}^\ast$,   $t \in [s]$,  we denote 
\[
\langle T,  \ell_{i_1} \otimes \cdots \otimes \ell_{i_s} \rangle \coloneqq \sum_{j=1}^r  \ell_{i_1}(w_{i_1,j})\cdots \ell_{i_s}(w_{i_s,j}) \medotimes_{1 \le i \le k,\;i \not\in \{i_1,\dots,  i_s\}} w_{i,j}.
\]
Moreover,  for any permutation $\sigma \in \mathfrak{S}_k$,  we identify $\medotimes_{i=1}^k \mathbb{W}_{\sigma(i)}$ with $\medotimes_{i=1}^k \mathbb{W}_i$ so that $w_{\sigma(1)} \otimes \cdots \otimes w_{\sigma(k)}$ is understood as $w_1 \otimes \cdots \otimes w_k$,  and this identification extends linearly.

\begin{definition}[Ranks of tensors]\label{def:rank}
Let $\mathbb{W}_1,\dots,  \mathbb{W}_k$ be vector spaces over $\mathsf{K}$.  Given $T \in \mathbb{W}_1 \otimes \cdots \otimes \mathbb{W}_k$,  we define its \emph{geometric rank} as 
\[
\GR(T) \coloneqq \codim\{(\ell_{1},\dots, \ell_{k-1})\in \mathbb{W}_1^\ast \times \cdots \times \mathbb{W}_{k-1}^\ast: 
\langle T,  \ell_1 \otimes \cdots \otimes \ell_{k-1} \rangle = 0\}.   
\]
A tensor of the form $S_1 \otimes S_2 \in \mathbb{W}_1 \otimes \cdots \otimes \mathbb{W}_k$ has partition rank one if $S_1 \in \mathbb{W}_{\sigma(1)} \otimes \cdots \mathbb{W}_{\sigma(p)}$,  $S_2 \in \mathbb{W}_{\sigma(p+1)} \otimes \cdots \mathbb{W}_{\sigma(k)}$ for some $\sigma \in \mathfrak{S}_k$ and $1 \le p \le k-1$.  The \emph{partition rank} of $T$ is defined by 
\[
\pr(T) \coloneqq \min \left\lbrace
r \in \mathbb{N}: T = \sum_{j=1}^r T_j,\; T_j \text{~has partition rank one}
\right\rbrace. 
\]
We denote by $I_r \in  (\mathsf{K}^r)^{\otimes k}$ the identity tensor of order $k$ and dimension $r$.  The \emph{subrank} of $T$ is
\[
\Q(T) \coloneqq \max \left\lbrace
r \in \mathbb{N}:  I_r \unlhd T
\right\rbrace.
\]
If $\mathsf{K}$ is a finite field,  the analytic rank of $T$ is
\[
\AR(T) \coloneqq \sum_{i=1}^{k-1} \dim \mathbb{W}_i - \log_{|\mathsf{K}|} \left\lvert
(\ell_1,\dots,  \ell_{k-1}) \in \mathbb{W}_1^\ast \times \cdots \mathbb{W}_{k-1}^\ast: \langle T,  \ell_1 \otimes \cdots \otimes \ell_{k-1} \rangle = 0
\right\rvert 
\]
\end{definition}
We remark that the four ranks in Definition~\ref{def:rank},  which may appear quite different at first glance, are in fact closely related. They will play a crucial role in our study of the isotropy and completeness indices to be defined in Section~\ref{sec:indices}.  For ease of reference,  we record some existing results below. 
\begin{lemma}[PR vs.  GR]
\label{prvsgr}
For any order $k$ tensor $T$ over a field $\mathsf{K}$,  we have the following:
\begin{enumerate}[label=(\alph*)]
\item \cite[Theorem~1]{cohen2021structure} If $\mathsf{K}$ is perfect and $k = 3$,  then $\PR(T) \asymp \GR(T)$.
\item \cite[Corollary~3]{cohen2023partition} If $\mathsf{K}$ is algebraically closed,  then $\PR(T) \asymp_k  \GR(T)$.
\item \cite[Theorem~2.3]{baily2024strength} If $\mathsf{K}$ is infinite,  then $\PR(T) \asymp_{k,n}  \GR(T)$ where $n$ is the transcendence degree of $\mathsf{K}$ over its prime field.
\item \cite[Theorem~1.6.3]{bik2025strength} If $\mathsf{K}$ is infinite,  then $\PR(T) \lesssim_k \GR(T)^{k-1}$.  
\end{enumerate}
\end{lemma} 
\begin{lemma}[PR vs. AR]\cite{moshkovitz2022quasi}\label{prvsar} 
For any tensor $T$ over a finite field $\mathsf{K}$ of order $k$,  we have $\PR(T) \lesssim_k   \AR(T) \log_{|\mathsf{K}|}(\AR(T))$.
\end{lemma}
\begin{lemma}[GR vs. AR]\cite{moshkovitz2024uniform, baily2024strength,chen2024stability}\label{grvsar}
For any tensor over a finite field of order $k$,  we have $ \AR(T) \asymp_k \GR(T)$.
\end{lemma}

\subsection{Algebraic geometry over finite fields}
Let $\mathsf{K}$ be a field and let $f \in \mathsf{K}[x_{1},\dots,x_N]$ be a polynomial.  For each $1 \le i \le N$,  we denote the degree of $f$ with respect to the variable $x_i$ by $\deg_{x_i} f$.  We also denote by $V(f)$ the zero set of $f$ in $\mathsf{K}^N$. 
\begin{lemma}[Combinatorial Nullstellensatz]\cite[Lemma~2.1]{alon1999combinatorial}\label{comb null}
Let $\mathsf{K}$ be a field and let $f \in \mathsf{K}[x_{1},\dots,x_N]$ be a nonzero polynomial.  For any subsets $S_1,  \dots,  S_N \subseteq \mathsf{K}$ such that $\vert S_{i} \vert> \deg_{x_{i}} f$ where $1 \le i \le N$,  we have $f(a) \ne 0$ for some $a \in S_{1}\times\cdots \times S_{N}$.
\end{lemma}

\begin{lemma}[Generalized Schwarz-Zippel lemma]\cite[Theorem~5]{Alon1993covering}\label{alon-furedi bound} 
Let $\mathsf{K}$ be a field.  Suppose that $Y_1,  \dots,  Y_N \subseteq \mathsf{K}$ are finite subsets with cardinality $r_1,\dots,  r_N$ respectively.  If $f \in \mathsf{K}[x_1,\dots,  x_N]$ is a degree $d$ polynomial that is not identically zero on $Y \coloneqq Y_1 \times \cdots \times Y_N \subseteq \mathsf{K}^N$,  then
\[
 \lvert Y \setminus V(f)  \rvert  \ge 
\min  \left\lbrace
z_1 \cdots z_N: z_i \in [r_i], \; i \in [N],  \;  z_1 + \cdots + z_N = r_1 + \cdots + r_N - d
\right\rbrace.
\]
\end{lemma}
If we take $Y_1 = \cdots = Y_N \subseteq \mathsf{K} = \mathsf{F}_q$ and assume $d < q$,  then Lemma~\ref{alon-furedi bound} reduces to a weaker version of the celebrated Schwartz-Zippel Lemma,  which is reproduced below for easy reference.   
\begin{lemma}[Schwartz-Zippel lemma]\cite{Schwartz80,Zippel79}
\label{Schwartz-Zippel lemma}
Let $R$ be an integral domain and let $S \subseteq R$ be a finite subset.  If $f\in R[x_{1},\dots,x_N]$ is nonzero and has degree $d$,  then $\vert V(f) \cap S^N \vert \le d |S|^{N-1}$.  Here $V(f)$ denotes the zero set of $f$ in $R^N$.
\end{lemma}
\section{Isotropy index and completeness index}\label{sec:indices}
In this section,  we define,  for each of alternating,  symmetric and general multilinear maps,  the isotropy index and the completeness index.  Given $F \in \Hom(\mathbb{V}_1 \times \cdots \times \mathbb{V}_{d},  \mathbb{W})$ and a subspace $\mathbb{U} \subseteq \mathbb{V}_1 \times \cdots \times \mathbb{V}_{d}$,  we denote by $F|_{\mathbb{U}} \in \Hom(\mathbb{U},\mathbb{W})$ the restriction of $F$ on $\mathbb{U}$.  We define 
\begin{equation}\label{eq:kerF}
\Ker(F) \coloneqq \lbrace 
T \in \mathbb{V}_1 \otimes \cdots \otimes \mathbb{V}_{d}: F(T) = 0
\rbrace,
\end{equation}
where $F$ is viewed as an element in $\Hom(\mathbb{V}_1 \otimes \cdots \otimes \mathbb{V}_{d},  \mathbb{W})\simeq  \Hom(\mathbb{V}_1 \times \cdots \times \mathbb{V}_{d},  \mathbb{W})$. 
\begin{definition}[Isotropy and completeness indices]\label{def:index}
Let $\mathbb{V}_1,\dots,  \mathbb{V}_d,  \mathbb{W}$ be vector spaces over $\mathsf{K}$ and let $F \in \Hom(\mathbb{V}_1 \times \cdots \times \mathbb{V}_{d},\mathbb{W})$. 
\begin{enumerate}[label=(\alph*)]
\item A \emph{totally isotropic subspace} of $F$ is a space $\mathbb{U}_1 \times \cdots \times \mathbb{U}_{d}$ such that $F|_{\mathbb{U}_1 \times \cdots \times \mathbb{U}_{d}} = 0$,  where $\mathbb{U}_1 \subseteq \mathbb{V}_1,\dots,  \mathbb{U}_{d} \subseteq \mathbb{V}_{d}$ are linear subspaces of the same dimension. The \emph{isotropy index} of $F$ is defined as 
\[
\alpha(F) \coloneqq \max \left\lbrace
s: F|_{\mathbb{U}_1 \times \cdots \times \mathbb{U}_{d}} = 0,\; \dim \mathbb{U}_1 = \cdots = \dim \mathbb{U}_{d} = s
\right\rbrace.
\]
If $\mathbb{V}_1 = \cdots = \mathbb{V}_{d} = \mathbb{V}$, and $F \in \Alt^{d}(\mathbb{V},\mathbb{W})$ or $F \in \Sym^{d}(\mathbb{V},\mathbb{W})$,  then a \emph{totally isotropic subspace} of $F$ is a subspace $\mathbb{U}^{d} \subseteq \mathbb{V}^{d}$ such that $F|_{\mathbb{U}^{d}} = 0$.  The \emph{isotropy index} of $F$ is $\max \left\lbrace
\dim \mathbb{U}: F|_{\mathbb{U}^{d}} = 0,\;  \mathbb{U} \subseteq \mathbb{V}
\right\rbrace$,  respectively denoted by $\alpha_{\w}(F)$ and $\alpha_{\s}(F)$. 

\item A \emph{complete subspace} of $F$ is a space $\mathbb{U}_1 \times \cdots \times \mathbb{U}_{d}$ such that $\Ker (F|_{\mathbb{U}_1 \times \cdots \times \mathbb{U}_{d}} ) = \{0\}$,  where $\mathbb{U}_1 \subseteq \mathbb{V}_1,\dots,  \mathbb{U}_{d} \subseteq \mathbb{V}_{d}$ are linear subspaces of the same dimension.  The \emph{completeness index} of $F$ is 
\[
\beta(F) \coloneqq \max \left\lbrace
s: \Ker(F|_{\mathbb{U}_1 \times \cdots \times \mathbb{U}_{d}}) = \{0\},\; \dim \mathbb{U}_1 = \cdots = \dim \mathbb{U}_{d} = s
\right\rbrace.
\]
If $\mathbb{V}_1 = \cdots = \mathbb{V}_{d} = \mathbb{V}$,  and $F \in \Alt^d(\mathbb{V},\mathbb{W})$ or $F \in \Sym^d (\mathbb{V},\mathbb{W})$,  then a \emph{complete subspace} of $F$ is a subspace $\mathbb{U}^{d} \subseteq \mathbb{V}^{d}$ such that $\Ker(F|_{\mathbb{U}^{d}}) = 0$.  The \emph{completeness index} of $F$ is respectively defined as 
\begin{align*}
\beta_{\w}(F) &\coloneqq \max \left\lbrace
\dim \mathbb{U}: \Ker(F|_{\mathbb{U}^{d}}) \cap \Lambda^{d} \mathbb{V} = \{0\},\;  \mathbb{U} \subseteq \mathbb{V}
\right\rbrace,  \\
\beta_{\s}(F) &\coloneqq \max \left\lbrace
\dim \mathbb{U}: \Ker(F|_{\mathbb{U}^{d}}) \cap \mathrm{S}^{d} \mathbb{V} = \{0\},\;  \mathbb{U} \subseteq \mathbb{V}
\right\rbrace.
\end{align*}
\end{enumerate}
\end{definition}
It is worth noticing that the isotropy and completeness indices quantify two extremal structures of multilinear maps.  For instance,  $\alpha(F)$ measures the size of subspaces of the form $\mathbb{U}_1 \otimes \cdots \otimes \mathbb{U}_d \subseteq \mathbb{V}_1 \otimes \cdots \otimes \mathbb{V}_d$ where $\dim \mathbb{U}_1 = \cdots = \mathbb{U}_d$,  such that $\mathbb{U}_1 \otimes \cdots \otimes \mathbb{U}_d \subseteq \Ker(F)$.  By contrast,  $\beta(F)$ measures the size of subspaces of the same form such that $\mathbb{U}_1 \otimes \cdots \otimes \mathbb{U}_d \cap  \Ker(F) = \{0\}$.
\begin{remark}
The isotropy and completeness indices have been extensively studied for $d \le 2$:
\begin{enumerate}[label = (\alph*)]
\item For each $\varphi \in  \Hom(\mathbb{V},\mathbb{W})$,  we clearly have $\beta(\varphi) = \rank (\varphi)$.
\item For $d = 2$ and $\mathbb{W} = \mathsf{K}$,  both indices are well-understood in linear algebra \cite{MH73,Lang02}.  
\item The isotropy index plays an essential role in the study of alternating multilinear maps,  and has been coined with different names in various contexts \cite{Atkinson73, buhler1987isotropic,  Dav1992number,DS10,BMW17,qiao2020tur,BCGQS21}. 
\item Given $F \in \Alt^{d}(\mathbb{V}, \mathbb{W})$,  it is straightforward to verify that $\mathbb{U} \subseteq \mathbb{V}$ is a complete subspace of $F$ if and only if $\dim \left( \spa_{\mathsf{K}}\im \left( F|_{\mathbb{U}^{d}} \right) \right) = \binom{s}{d}$.  Thus,  our definition of complete subspaces coincides with the one in \cite[Subsection~2.4]{qiao2020tur} when $d = 2$. 
\end{enumerate}
\end{remark}

Let $R \coloneqq \mathsf{K}[x_1,\dots,  x_m]$ be the polynomial ring over $\mathsf{K}$ with $m$ variables.  We denote by $R_k$ the subspace of $R$ consisting of degree $k$ homogeneous polynomials.  Suppose either $\ch(\mathsf{K}) = 0$ or $\ch(\mathsf{K}) > d$.  We recall that symmetric $d$-multilinear maps are in $1$--$1$ correspondence with degree $d$ homogeneous polynomials via the polarization map \cite[Chapter~3, Section~2]{procesi2007lie}:
\begin{equation}\label{eq:polar}
\mathcal{P}: \Sym^d(\mathsf{K}^m,  \mathsf{K}) \to R_d,\quad \mathcal{P}(F)(x_1,\dots,  x_n) \coloneqq F(x,\dots,  x),
\end{equation}
where $x \coloneqq (x_1,\dots,  x_m)$.  That being so,  every $F  \in \Sym^d(\mathsf{K}^m,  \mathsf{K}^n)$ uniquely determines a subspace $\mathbb{L}_F \subseteq R_d$.  Indeed,  if we write $F = (F_1,\dots,  F_n)$ where $F_1,\dots,  F_n \in \Sym^d(\mathsf{K}^m,  \mathsf{K})$,  then
\begin{equation}\label{eq:LF}
\mathbb{L}_F = \spa \lbrace \mathcal{P}(F_1),\dots,  \mathcal{P}(F_n) \rbrace,  \quad \dim \mathbb{L}_F = n.
\end{equation}
The connection between $F$ and $\mathbb{L}_F$ leads to the following simple observation,  which is recorded here for subsequent reference.
\begin{proposition}[Isotropy and completeness indices for polynomial subspace]\label{lem:poly=sym}
Let $\mathbb{V},  \mathbb{W}$ be vector spaces over a field $\mathsf{K}$ and let $R$ be the ring of polynomials with $m \coloneqq \dim \mathbb{V}$ variables.   Suppose either $\ch(\mathsf{K}) = 0$ or $\ch(\mathsf{K}) > d$.  Given $F \in \Sym^d(\mathbb{V},  \mathbb{W})$ and a subspace $\mathbb{U} \subseteq \mathbb{V}$,  
\begin{enumerate}[label = (\alph*)]
\item $\mathbb{U}^d$ is a totally isotropic subspace of $F$ if and only if $\mathbb{L}_F \subseteq \mathfrak{a}(\mathbb{U})_d$,  which is further equivalent to $\mathfrak{a}_F \subseteq \mathfrak{a}(\mathbb{U})^d$.
\item $\mathbb{U}^d$ is a complete subspace of $F$ if and only if the quotient map $\pi: R_d \to (R/\mathfrak{a}(\mathbb{U}))_d$ is surjective on $\mathbb{L}_{F}$.
\end{enumerate}
In particular,  we have 
\begin{align*}
\alpha_{\s} (F) = \alpha_{\p} (\mathbb{L}_F) &\coloneqq \max \left\lbrace
s \in \mathbb{N}: \mathbb{L}_{F} \subseteq  (\ell_1,\dots,  \ell_{m-s})_d  \text{~for some~}\ell_1,\dots,  \ell_{m-s} \in R_1
\right\rbrace.   \\
&= \max \left\lbrace
s \in \mathbb{N}: \mathfrak{a}_F \subseteq  (\ell_1,\dots,  \ell_{m-s})^d  \text{~for some~}\ell_1,\dots,  \ell_{m-s} \in R_1
\right\rbrace.  \\
\beta_{\s} (F) = \beta_{\p} (\mathbb{L}_F) &\coloneqq  \max \left\lbrace
s \in \mathbb{N}: \pi(\mathbb{L}_{F}) = (R/(\ell_1,\dots,  \ell_{m-s}))_d  \text{~for some~}\ell_1,\dots,  \ell_{m-s} \in R_1
\right\rbrace.
\end{align*}
Here $\mathfrak{a}(\mathbb{U})$ is the defining ideal of  $\mathbb{U}$,  $\mathfrak{a}_F$ is the ideal generated by $\mathbb{L}_F$,  $\mathfrak{m}$ is the maximal ideal of $R$,  and $R_k$ (resp.  $\mathfrak{a}_d$) denotes the degree $k$ part of $R$ (resp. $\mathfrak{a}$) for each $k \in \mathbb{N}$.  
\end{proposition}

In what follows,  we establish some basic properties of the isotropy and completeness indices. 
\begin{lemma}[Basic properties of isotropy and completeness indices]\label{lem:basic}
Given $d$-multilinear maps $F \in \Hom(\mathbb{V}_1 \times \cdots \times \mathbb{V}_{d},\mathbb{W})$ and $F' \in \Hom(\mathbb{V}'_1 \times \cdots \times \mathbb{V}'_{d},\mathbb{W}')$,  we have 
\begin{enumerate}[label = (\alph*)]
\item\label{lem:basic:item1} $\alpha(F) + \alpha(F') \le \alpha(F \oplus F')$.
\item\label{lem:basic:item2} $\alpha(F)^d \le \dim \left( \Ker (F) \right)$ and $\beta(F)^d  \le \dim \left( \spa F(\mathbb{V}_1 \times \cdots \times \mathbb{V}_d) \right)$.
\item\label{lem:basic:item3} $\max \{ \beta(F),  \beta(F')\} \le \beta(F \oplus F')$.
\end{enumerate}
Moreover,  if $\mathbb{V}_1 = \cdots = \mathbb{V}_d = \mathbb{V}$,  then  
\begin{enumerate}[label = (\alph*')]
\item\label{lem:basic:item1'}
$\alpha_{\w}(F) + \alpha_{\w}(F') = \alpha_{\w}(F \oplus F')$ if $F$ and $F'$ are alternating multilinear maps,  and $\alpha_{\s}(F)+\alpha_{\s}(F')=\alpha_{\s}(F \oplus F')$ if $F$ and $F'$ are symmetric multilinear maps.
\item\label{lem:basic:item2'} $\binom{\alpha_{\w}(F)}{d} \le \dim \left( \Ker (F) \cap \Lambda^d \mathbb{V} \right)$ and $\binom{\beta_{\w}(F)}{d} \le \dim \left( \spa F( \mathbb{V}^d ) \right)$ for $F \in \Alt^d(\mathbb{V},\mathbb{W})$; $\binom{\alpha_{\s}(F) + d  -1}{d} \le \dim \left( \Ker (F) \cap \mathrm{S}^d \mathbb{V} \right)$ and $\binom{\beta_{\s}(F) +d -1}{d}  \le \dim \left( \spa F( \mathbb{V}^d ) \right)$ for $F \in \Sym^d(\mathbb{V},\mathbb{W})$.
\end{enumerate}
\end{lemma}
\begin{proof}
It is clear that \ref{lem:basic:item2},  \ref{lem:basic:item3} and \ref{lem:basic:item2'} follow immediately from the definition.  To prove \ref{lem:basic:item1},  we denote by $\mathbb{U} \coloneqq \smallprod_{i=1}^d \mathbb{U}_i$ and $\mathbb{U}' \coloneqq \smallprod_{i=1}^d \mathbb{U}'_i$ maximal totally isotropic subspaces of $F$ and $F'$,  respectively.  Note that
\[
(F\oplus F') \left( v_1 \oplus v_1', \dots,  v_d \oplus v'_d \right) =F(v_1,\dots,  v_d) \oplus  F'(v'_1,\dots,  v'_d) = 0
\]
for any $\left( v_1 \oplus v_1', \dots,  v_d \oplus v'_d \right)  \in \smallprod_{i=1}^d (\mathbb{U}_i \oplus \mathbb{U}'_i)$.  Thus,  we have $\alpha (F) + \alpha(F') \le \alpha(F \oplus F')$.  Similarly,  we have $\alpha_{\w}(F) + \alpha_{\w}(F') \le \alpha_{\w}(F \oplus F')$ (resp.  $\alpha_{\s}(F) + \alpha_{\s}(F') \le \alpha_{\s}(F \oplus F')$) if $F \in \Alt^d (\mathbb{V},  \mathbb{W})$ and $F' \in \Alt^d (\mathbb{V}, ' \mathbb{W})$ (resp.  $F \in \Sym^d (\mathbb{V},  \mathbb{W})$ and $F' \in \Sym^d (\mathbb{V}, ' \mathbb{W})$ ).

Let $\pi:\mathbb{V}\oplus \mathbb{V}'\to \mathbb{V}$ and $\pi':\mathbb{V}\oplus \mathbb{V}'\to \mathbb{V}'$ be the natural projection maps.  Suppose that $(F\oplus F')|_{\mathbb{U}}=0$ and $\alpha_{\Lambda}(F\oplus F')= \dim \mathbb{U}$ for some $\mathbb{U}\subseteq \mathbb{V}\oplus\mathbb{V}'$.  This implies $F|_{\pi_{1}(\mathbb{U})}=0$ and $F'|_{\pi_{2}(\mathbb{U})}=0$.  Hence $\alpha_{\w}(F)\ge \dim \pi_{1}(\mathbb{U})$ and $\alpha_{\w}(F')\ge \dim \pi_{2}(\mathbb{U})$.  By definition,  it holds that $\mathbb{U}\subseteq \pi_{1}(\mathbb{U})\oplus \pi_{2}(\mathbb{U})$.  Thus,  we must have 
\[
\alpha_{\w}(F)+\alpha_{\w}(F') \ge \dim(\pi_{1}(\mathbb{U}))+ \dim(\pi_{2}(\mathbb{U}))\ge \dim \mathbb{U} = \alpha_{\Lambda}(F\oplus F'),
\]
from which we deduce $\alpha_{\w}(F) + \alpha_{\w}(F') = \alpha_{\w}(F \oplus F')$.  The same argument may apply to show the additivity of $\alpha_{\s}$.
\end{proof}
For special multilinear maps,  the isotropy and completeness indices can be computed explicitly.  Below we consider those arising from three familiar algebras: the split algebra,  the matrix algebra,  and the group algebra.  To this end, we will need the following observation.
\begin{lemma}\label{support set} 
Let $\mathbb{U}$ be an $s$-dimensional subspace of $\mathsf{K}^{n}$.  Then there is some $a = (a_1,\dots,  a_n) \in \mathbb{U}$ such that $s \le \left\lvert \left\lbrace  
i \in [n]: a_i \ne 0
\right\rbrace \right\rvert$.
\end{lemma}
\begin{proof}
Let $u_1,\dots,  u_s \in \mathbb{U}$ be a basis of $\mathbb{U}$.  For each $j \in [s]$,  we write $u_j = (u_{1,j},\dots,  u_{n,j})^\tp \in \mathsf{K}^n$ and denote $A \coloneqq (u_{i,j})_{(i,j) \in [n] \times [s]} \in \mathsf{K}^{n \times s}$.  It is clear that $\rank A = s$.  Without loss of generality,  we may assume that the first $s$ rows of $A$ are linearly independent.  We consider 
\[
\pi: \mathsf{K}^s \to \mathsf{K}^s,\quad \pi(c) = \begin{bmatrix}
(Ac)_1 \\
\vdots \\
(Ac)_s
\end{bmatrix}.
\]
Since the matrix of $\pi$ is the submatrix of $A$ consisting of its first $s$ rows,  $\pi$ is an isomorphism.  Thus,  there is some $c \in \mathsf{K}^s$ such that the first $s$ coordinates of $a \coloneqq Ac \in \mathbb{U}$ are nonzero. 
\end{proof}
\begin{proposition}[Typical examples]\label{prop:ex}
Suppose $\mathsf{K}$ is a field. 
\begin{enumerate}[label = (\alph*)]
\item\label{prop:ex:item1} For any positive integers $r,d$,  we define $I_{r,d}: \underbrace{\mathsf{K}^r \times \cdots \times \mathsf{K}^r}_{d \text{~times}} \to \mathsf{K}^r$ by
\begin{equation}\label{prop:ex:item1:eq1}
I_{r,d} \left( (a_{1,1},\dots,  a_{1,r}),\dots,  (a_{d,1},\dots,  a_{d,r}) \right) = \left( \smallprod_{i=1}^d a_{i,1},\dots,  \smallprod_{i=1}^d a_{i,r} \right).
\end{equation}
Then $\alpha(I_{r,d})  = \lfloor r (1- 1/d) \rfloor$,  $\alpha_{\s}(I_{r,d})  = 0$ and $\beta(I_{r,d}) =  \lfloor r^{1/d}\rfloor$.  If moreover, $\vert\mathsf{K}\vert>d$,  then we have  
\[
\beta_{\s}(I_{r,d})= \max\left\lbrace m \in\mathbb{N}: \binom{m + d - 1}{d} \le r\right\rbrace.
\]
\item\label{prop:ex:item2} For each positive integer $n$,  we denote by $M_n: \mathsf{K}^{n\times n} \times \mathsf{K}^{n \times n} \to \mathsf{K}^{n\times n}$ the matrix multiplication map.  Then we have $\alpha(M_n) = \lfloor \frac{n}{2} \rfloor n$ and $\beta(M_n) = n$.
\item\label{prop:ex:item3} Suppose $\mathsf{K}$ is algebraically closed and $G$ is a finite group such that $\ch(\mathsf{K})\nmid |G|$.  Let $\varphi_G: \mathsf{K}[G] \times \mathsf{K}[G] \to \mathsf{K}[G]$ be the multiplication map of the group algebra of $G$.  Then
\begin{align*}
 \alpha (\varphi_G) &\ge \left\lfloor \frac{\lvert \lambda \in \Irr(G): \dim \lambda = 1 \rvert}{2} \right\rfloor + \sum_{\lambda \in \Irr(G),\; \dim \lambda > 1} \left\lfloor \frac{n_{\lambda}}{2} \right\rfloor n_{\lambda},\\
\end{align*}
Here $\Irr(G)$ is the set of isomorphism classes of irreducible representations of $G$,  and $n_{\lambda}$ is the dimension of $\lambda \in \Irr(G)$.  Moreover,  the equality holds if $G$ is abelian.
\end{enumerate}
\end{proposition}
\begin{proof}
\begin{itemize}
\item[$\diamond$] We first prove \ref{prop:ex:item1}.  Since $\mathsf{K}$ is a field,  it is obvious that $\alpha_{\s}(I_{r,d}) = 0$.  Given $v_1,\dots,  v_d \in \mathsf{K}^r$,  $I_{r,d} (v_1,\dots,  v_d) = 0$ implies that for each $j \in [r]$,  there is some $i_j \in [d]$ such that $v_{i_j,  j} = 0$ where $v_{i,k}$ is the $k$-th element of $v_i$,  $i \in [d],  k \in [r]$.  If $\alpha( (I_{r,d} ) > r - r/d$, there exists $\mathbb{U}_{i}\subseteq \mathsf{K}^{r}$ such that $\dim \mathbb{U}_i = s > r - r/d$ and $I_{r,d} (v_1,\dots,  v_d)=0$ for any $v_{i}\in\mathbb{U}_{i}$ where $i \in [d]$. Lemma~\ref{support set} ensures the existence of $u_{1}\in\mathbb{U}_1,\dots,  u_d \in \mathbb{U}_d$,  each of which has at least $s$ nonzero coordinates.  By the pigeonhole principle,  we deduce a contradiction that $I_{r,d}(u_{1},\cdots,u_{d}) \ne 0$,  and this implies $\alpha( I_{r,d} ) \le r - r/d$.  To show that the upper bound can be achieved,  we consider for each $i \in [d]$,  the subspace $\mathbb{U}_i$ consisting of all $v = (v_1,\dots,  v_r) \in \mathsf{K}^r$ with $v_{(i-1)s + 1} = \cdots = v_{is} = 0$ where $s \coloneqq r - \lfloor r (1- 1/d) \rfloor$,  and the indices $(i-1)s + 1,\dots,  is$ are taken modulo $r$.  It is clear that $\smallprod_{i=1}^d \mathbb{U}_i$ is a totally isotropic subspace of $I_{r,d}$.  

By Lemma~\ref{lem:basic}--\ref{lem:basic:item2},  we have $\beta(I_{r,d})\le b \coloneqq \lfloor r^{1/d}\rfloor$.  It is left to prove that the upper bound can be achieved.   Denote $b \coloneqq \lfloor r^{1/d}\rfloor$ and define for each $i \in [d]$ and $0 \le j \le b - 1$ a vector $v_{i,j} = (v_{i,j,1},\dots,  v_{i,j,r}) \in \mathsf{K}^r$,  where
\[
v_{i,j,k} \coloneqq \begin{cases}
1 \quad &\text{if the $i$-th term of the $b$-adic expansion of $k$ is $j$}, \\
0 \quad &\text{otherwise}.
\end{cases}
\]
By construction,  $I_{r,d} (v_{1,j_1},\dots,  v_{d,j_d})$ is the vector in $\mathsf{K}^r$ whose elements are all zero except for the $(\sum_{s = 0}^{d-1} j_s b^s)$-th,  which is equal to $1$.  This implies that $\mathbb{U}_1 \times \cdots \times \mathbb{U}_d$ is a complete subspace,  where $\mathbb{U}_i $ is the $b$-dimensional subspace spanned by $v_{i,0},\dots,  v_{i,b-1}$ for $i \in [d]$. 

Next,  let  
\[
\beta \coloneqq \max
\left\lbrace b\in\mathbb{N}: \binom{b + d - 1}{d} \le r\right\rbrace,\quad N \coloneqq \binom{\beta + d - 1}{d}.
\]
By Lemma~\ref{lem:basic}--\ref{lem:basic:item2'},  we clearly have $\beta_{\s}(I_{r,d}) \le \beta$.  We consider the map
\[
\nu:\mathsf{K}^{\beta}\to  \mathsf{K}^N,\quad \nu (v_1,\dots, v_\beta) = (v_1^{m_1}\cdots v_\beta^{m_{\beta}}),
\]
where $(m_1,\dots,  m_\beta) \in \mathbb{N}^{\beta}$ and $m_1 + \cdots + m_{\beta} = d$.  We notice that the image of $\nu$ spans $\mathsf{S}^N \simeq \mathsf{S}^d \mathsf{K}^{\beta}$.  Otherwise,  there is some degree $d$ homogeneous polynomial $\lambda \in \mathsf{S}^d ( \mathsf{K}^{\beta} )^\ast$ such that $\lambda(v) = \langle \lambda,  \nu(v)\rangle = 0$ for any $v \in \mathsf{K}^{\beta}$.  This contradicts to the fact that $\vert\mathsf{K}\vert>d$ by Lemma~\ref{comb null}.  Since the image of $\nu$ spans $\mathsf{K}^N$,  there exist $v_1,\dots,  v_N \in \mathsf{K}^{\beta}$ such that $\nu(v_1),\dots,  \nu (v_{N}) \in \mathsf{K}^N$ are linearly independent.  We write 
\[
M \coloneqq \begin{bmatrix}
v_1 & \cdots & v_N & 0_{\beta \times (r - N)}
\end{bmatrix} \in \mathsf{K}^{\beta \times r}.
\]
Here $0_{\beta \times (r - N)}$ denotes the zero matrix of size $\beta \times (r-N)$.  If $\rank M < \beta$,  then $v_1,\dots,  v_N$ are contained in a proper subspace of $\mathsf{K}^\beta$.  This implies that $\nu(v_1),\dots,  \nu(v_N)$ must be linearly dependent.  Thus,  we have $\rank M = \beta$ and $\mathbb{U} \coloneqq \spa \{ u_1 ,\dots,  u_\beta \}$ is $\beta$-dimensional where $u_1,\dots,  u_\beta \in \mathsf{K}^r$ are row vectors of $M$.  Note that the row vectors of 
\[
\begin{bmatrix}
\nu(v_1)& \cdots & \nu(v_N) & 0_{N \times (r-N)}\\
\end{bmatrix} \in \mathsf{K}^{N \times N}
\]
are exactly $I_{r,d} (u_{i_1},\dots,  u_{i_d})$ where $1 \le i_1 \le \cdots \le i_d \le \beta$.  The linear independence of $\nu(v_1),\dots,  \nu(v_N)$ implies that $\mathbb{U}$ is a complete subspace of $I_{r,d}$.  
\item[$\diamond$] Next,  we consider \ref{prop:ex:item2}.  Suppose that $\mathbb{U}_1 \times \mathbb{U}_2$ is a maximal totally isotropic subspace of $M_n$.  If the maximal rank of matrices in $\mathbb{U}_1$ is $r$,  then each matrix in $\mathbb{U}_2$ has rank at most $n -r$.  Without loss of generality,  we suppose that $r \le n/2$.  Then by \cite[Theorem~1]{Flanders62},  we have 
\[
\alpha(M_n) = \dim \mathbb{U}_1 = \dim \mathbb{U}_2 \le rn \le \left\lfloor \frac{n}{2} \right\rfloor n.
\]
The upper bound can be achieve by 
\[
\mathbb{U}_1 \coloneqq \left\lbrace
\begin{bmatrix}
0 & A
\end{bmatrix} \in \mathsf{K}^{n \times n}: A \in \mathsf{K}^{n \times \lfloor \frac{n}{2} \rfloor}
\right\rbrace,\quad 
\mathbb{U}_2 \coloneqq \left\lbrace
\begin{bmatrix}
B \\
0
\end{bmatrix} \in \mathsf{K}^{n \times n}: B \in \mathsf{K}^{\lfloor \frac{n}{2} \rfloor \times n}
\right\rbrace.
\]
By Lemma~\ref{lem:basic}--\ref{lem:basic:item3},  we have $\beta(M_n) \le n$.  We let $\mathbb{V}_1 $ be the space of $n \times n$ diagonal matrices and $\mathbb{V}_2$ be the space of $n \times n$ circulant matrices.  For any $i,  j \in [n]$,  we have $D_i C_{j-i} = E_{i,j}$,  which implies $\mathbb{V}_1 \times \mathbb{V}_2$ is a complete subspace of $M_n$ and completes the proof.  Here $D_s$ is the diagonal matrix whose diagonal elements are all zero,  except for the $s$-th one,  which is one;  $C_{s}$ is the circulant matrix whose elements in the first row are all zero,  except for the $s$-th,  which equals one; $E_{s,t}$ is the $n\times n$ matrix whose elements are all zero,  except for the $(s,t)$-th,  which equals one.  
\item[$\diamond$] By Wedderburn's theorem,  \ref{prop:ex:item3} is a direct consequence of \ref{prop:ex:item1},  \ref{prop:ex:item2} and Lemma~\ref{lem:basic}.  \qedhere
\end{itemize}
\end{proof}
\begin{remark}
According to Lemma~\ref{lem:basic}--\ref{lem:basic:item1'},  both $\alpha_{\w}$ and $\alpha_{\s}$ are additive with respect to the direct sum.  By contrast,  Proposition~\ref{prop:ex}--\ref{prop:ex:item1} implies that $\alpha$ is not additive.  It also shows that $\beta$ and $\beta_{s}$ are not additive.  Moreover,  we recall that $\rank (M_n) = \Omega ( \rank (I_{n^2,2})^2)$,  whereas $\beta(I_{n^2,2}) = \beta (M_n)$ and $\alpha(I_{n^2,2})\asymp \alpha(M_{n})$.  Here $\rank (F)$ is the CP-rank of a bilinear map $F$. This suggests that the isotropy and completeness indices behave in a dramatically different way from the CP-rank.
\end{remark}
We conclude this section by emphasizing that the isotropy and completeness indices are invariants of multilinear maps,  in the sense that they are invariant under the canonical action of $\GL(\mathbb{V}_1) \times \cdots \GL(\mathbb{V}_d) \times \GL(\mathbb{W})$ on $\Hom(\mathbb{V}_1 \times \cdots \times \mathbb{V}_d,  \mathbb{W})$.  However,  these indices are not invariants of tensors.  In fact,  they are not even well-defined for tensors.  For example,  a tensor $T \in \mathbb{V}_1 \otimes \mathbb{V}_2 \otimes  \mathbb{V}_3$ uniquely determines $F_{1} \in \Hom(\mathbb{V}_2^\ast \times \mathbb{V}_3^\ast,  \mathbb{V}_1)$ and $F_{2} \in \Hom(\mathbb{V}_1^\ast \times \mathbb{V}_3^\ast,  \mathbb{V}_2)$,  but it is clear that $\alpha(F_1) \ne \alpha(F_2)$ and $\beta(F_1) \ne \beta(F_2)$  in general.

\section{Algebraic technicalities}\label{sec:alg}
The goal of this section is to establish some non-vanishing results for maximal minors of matrices over polynomial rings.  Main results are Propositions~\ref{coro1} and \ref{lem2},  which will be used to connect the isotropy and completeness indices with tensor ranks in Section~\ref{sec:estimate}.  The idea that underlies the proof of the following lemma is borrowed from the proof of Theorem~2.4 in \cite{ananyan2020small}.
\begin{lemma}\label{lem1}
Let $R$ be a polynomial ring over a field.  Suppose $M = (M_{k,j})_{(k,j) \in [m+1] \times [n]} \in R^{(m + 1) \times n}$ satisfies the following conditions:
\begin{enumerate}[label=(\alph*)]
\item\label{lem1:item1} For each $k \in [m+1]$,  $M_{k,1},\dots,  M_{k,n}$ are homogeneous of the same degree $d_k$.  Moreover,  $ d_i < d_{m + 1}$ for each $ i \in [m]$.
\item\label{lem1:item2} The first $m$ rows of $M$ are linearly independent over the field of fractions of $R$.
\item\label{lem1:item3} $\height (\mathfrak{a}) \ge m +1$,  where $\mathfrak{a}$ is the ideal generated by $M_{m+1,1},\dots,  M_{m+1,n}$.
\end{enumerate}
Then $M$ has a nonzero maximal minor.
\end{lemma}
\begin{proof}
We begin by reducing the problem to a special case.  Suppose the ground field of $R$ is $\mathsf{K}$.  If $\mathsf{K}$ is not algebraically closed,  we consider the polynomial ring $\overline{R} \coloneqq R \otimes_{\mathsf{K}} \overline{\mathsf{K}}$ over $\overline{\mathsf{K}}$.  According to \cite[Lemma~3.4~(a)]{chen2025bounds},  we have $\height(\mathfrak{b} \overline{R} ) = \height(\mathfrak{b})$ for any ideal $\mathfrak{b}$ of $R$.  Thus any $M \in R^{(m+1) \times n}$ satisfying \ref{lem1:item1}--\ref{lem1:item3} for $R$ also satisfies \ref{lem1:item1}--\ref{lem1:item3} for $\overline{R}$.  Consequently,  it is sufficient to assume that $\mathsf{K}$ is algebraically closed.  

By \ref{lem1:item2},  $M$ has a nonsingular $m \times m$ submatrix,  and we may permute columns of $M$ so that it is the upper left $m \times m$ submatrix.  Let $\Delta$ be the determinant of this submatrix.  If there is a linear form $\ell \in R$ such that $\ell \nmid \Delta$ and $\ell \not\in \cup_{\mathfrak{p} \in \Ass(\mathfrak{a})} \mathfrak{p}$,  then $R/(\ell)$ is a polynomial ring of dimension $\dim R - 1$.  Here $\Ass(\mathfrak{a})$ is the set of associated primes of $\mathfrak{a}$. Moreover,  we have $\height(\pi_1 (\mathfrak{a}) ) = \height(\mathfrak{a})$ and $\pi_1 (\Delta) \ne 0$,  where $\pi_1: R \to R/(\ell)$ is the natural quotient map.  Since $\pi_1$ preserves the degree,  the matrix $(\pi_1(M_{ij})) \in \left( R/(\ell) \right)^{(m+1) \times n}$ satisfies \ref{lem1:item1}.  We may repeat the process to obtain a polynomial ring $R'$ over $\mathsf{K}$ together with a quotient map $\pi: R \to R'$,   such that there is no linear form $\ell \in R'$ simultaneously satisfying $\ell \nmid \Delta' \coloneqq \pi (\Delta) $ and $\ell \not\in \cup_{\mathfrak{p} \in \Ass( \mathfrak{a}' )} \mathfrak{p}$,  where $ \mathfrak{a}' \coloneqq \pi (\mathfrak{a})$.  By construction,  we have $\Delta'  \ne 0$,  $\height(\mathfrak{a}' ) = \height ( \mathfrak{a} )$ and $M' \coloneqq (\pi (M_{ij})) \in {R'}^{(m+1) \times n}$ satisfies \ref{lem1:item1}.  Clearly,  if $M'$ has a nonzero maximal minor,  so does $M$.

In the rest of the proof,  we assume that there is no linear form $\ell \in R$ simultaneously satisfying $\ell \nmid \Delta$ and $\ell \not\in \cup_{\mathfrak{p} \in \Ass( \mathfrak{a} )} \mathfrak{p}$.  Since $\Delta$ only has finitely many linear factors,  let $\ell_{1},\dots,\ell_{s}$ denote the linear factors of $\Delta$. Let $\mathfrak{m}$ be the maximal homogeneous ideal of $R$ and let $\mathfrak{m}_1$ be the linear space consisting of all linear forms in $R$.  Then we may deduce from the assumption that $\mathfrak{m}_1 \subseteq \{\ell_{1},\cdots,\ell_{s}\}\cup\cup_{\mathfrak{p} \in \Ass( \mathfrak{a} )} \mathfrak{p}$.  Since $\mathfrak{m}_1$ is a vector space,  we conclude that $\mathfrak{m}_1 \subseteq \cup_{\mathfrak{p} \in \Ass( \mathfrak{a} )} (\mathfrak{p})$.  As a consequence,  we have $\mathfrak{m} \in \Ass( \mathfrak{a} )$.  By definition,  there is some $f \in R\setminus \mathfrak{a}$ such that $f \mathfrak{m} \subseteq \mathfrak{a}$.  This implies $ \Ass( \mathfrak{a} ) = \{ \mathfrak{m} \}$ and $N = \height ( \mathfrak{m}) = \height ( \mathfrak{a})  \ge m + 1$.  

We prove the existence of a nonzero maximal minor of $M$ by contradiction.  To achieve the goal,  we let $M^{(i)}$ be the $m \times m$ submatrix of $M$ obtained by taking the first $m$ columns and removing the $i$-th row,  where $i \in [m]$. Denote $\Delta_i \coloneqq \det M^{(i)}$.  If all maximal minors of $M$ are zero,  then it must hold that $(\Delta_{1},\cdots,\Delta_{m},\Delta)M =0$.  Hence,  we obtain 
\[
\rho_{m+1}= \sum_{i=1}^{m}u_{i} \rho_{i},\quad u_i \coloneqq -\frac{\Delta_{i}}{\Delta},\;  i \in [m]. 
\]
Here $\rho_k$ is the $k$-th row of $M$ for $k \in [m+1]$.  By definition,  $\deg u_i = d_{m+1} - d_i > 0$,  $i \in [m]$.  

Next,  we consider the graded,  finitely generated $\mathsf{K}$-algebra $S \coloneqq R[u_{1},\cdots,u_{m}]$.  Since $R \subseteq S \subseteq \Frac(R)$,  we have by \cite[Chapter~4,  \S 15]{Matsumura70} that 
\[
m+1 \le N = \dim R = \operatorname{tr.  deg.}  \Frac(R) = \operatorname{tr. deg.} \Frac(S) = \dim S,
\] 
where $\Frac(R_1)$ is the field of fractions of an integral domain $R_1$ and $\operatorname{tr.  deg.} \mathsf{F}$ is the transcendence degree of a field extension $\mathsf{F}/\mathsf{K}$.  

To obtain a contradiction,  we denote by $ \mathfrak{c}$ the ideal of $S$ generated by $u_1,\dots,  u_m$.  Since $\mathfrak{a}$ is generated by elements of $\rho_{m+1}$,  we have $\mathfrak{a} S \subseteq  \mathfrak{c} $.  Moreover,  since $\deg u_i >0$ for each $i \in [m]$,  $\mathfrak{c} $ is a proper ideal of $S$.  Now that $f \mathfrak{m}  \subseteq \mathfrak{a}$ and $S$ is a finitely generated algebra over $\mathsf{K}$,  every prime ideal $\mathfrak{q}$ in $S$ containing $\mathfrak{c}$ must also contain $\mathfrak{m}$.  Hence $\Ass (\mathfrak{c}) = \{ \mathfrak{m} S + \mathfrak{c} \}$ as $\mathfrak{m} S + \mathfrak{c}$ is the maximal homogeneous ideal of $S$.  As a consequence,  we have $\height \mathfrak{c} = \height (\mathfrak{m} S + \mathfrak{c}) = \dim S = N \ge m+1$.  However,  by Krull’s principal ideal theorem \cite[Theorem~7.5]{kemper2011course},  this contradicts to the fact that $\mathfrak{c}$ is generated by $m$ elements in $S$.  
\end{proof}

Given a matrix $M$ over the polynomial ring $R = \mathsf{K}[x_1,\dots,  x_N]$ and $a \in \mathsf{K}^N$,  we denote by $M(a)$ the matrix over $\mathsf{K}$ obtained by evaluating elements of $M$ at $a$.
\begin{proposition}[Non-vanishing maximal minor I]\label{coro1} 
Let $\mathsf{K}$ be a field and let $R = \mathsf{K}[x_1,\dots,  x_N]$.  If a matrix $M \in R^{(m + 1) \times n}$ satisfies conditions \ref{lem1:item1}--\ref{lem1:item3} of Lemma~\ref{lem1} and $|\mathsf{K}| >\max_{s \in [N]} \sum_{i=1}^{m+1}d_{i,s} $,  then there exists $a \in \mathsf{K}^N$ such that $\rank M(a) = m+1$.  Here $d_{i,s} \coloneqq \max_{j \in [n]} \{ \deg_{x_s} M_{i,j} \}$ for $(i,s)\in [m+1] \times [N]$.
\end{proposition}
\begin{proof}
By Lemma~\ref{lem1},  $M$ has a nonzero maximal minor,  which is a polynomial whose degree in $x_s$ is at most $\sum_{i=1}^{m+1}d_{i,s}$ for each $s \in [N]$.  Since $|\mathsf{K}| >\max_{s \in [N]} \sum_{i=1}^{m+1}d_{i,s} $,  Lemma~\ref{comb null} implies the existence of a desired point in $\mathsf{K}^{N}$.
\end{proof}

To establish Proposition~\ref{lem2},  we need the following elementary result.
\begin{lemma}\label{adjustment}
Assume $N,  c,   q$ are integers such that $N \le c \le qN$.  The optimization problem
\begin{equation}\label{eq:opt}
\begin{tabular}{rl}
minimize & $z_1 \cdots z_N$\\
subject to & $z_{1}+\cdots+z_{N}= c$,  \\
& $z_i \in [q]$, \; $i \in [N]$.
\end{tabular}
\end{equation}
has a solution 
\[
z_1  = \cdots = z_{\lfloor \frac{c - N}{q-1} \rfloor} = q,  \;z_{\lfloor \frac{c - N}{q-1} \rfloor + 1} = \cdots = z_{N-1} = 1, \;z_{N} = c - N + 1 - (q-1) \left\lfloor \frac{c - N}{q-1} \right\rfloor.
\]
\end{lemma}
\begin{proof}
Suppose $(a_1,\dots,  a_N)$ is a solution of \eqref{eq:opt}.  Since $(x - 1)(y + 1) < xy$ for any integers $x \le y$,  there is at most one $i \in [N]$ such that $2 \le a_i \le q-1$.  Without loss of generality,  we assume $i = N$ and
\[
a_1 = \cdots = a_t = q,\; a_{t+1} = \cdots = a_{N-1} = 1
\]
for some $t \in [N-1]$.  The constraint that $z_1 + \cdots + z_N = c$ implies $(q-1)t + a_N = c - N + 1$,  from which we obtain $(c - N)/(q-1) - 1 \le t \le (c - N)/(q-1)$.  If $(q - 1) \nmid (c - N)$,  then $t = \lfloor (c - N)/(q-1) \rfloor$.  If $(q - 1) \mid (c - N)$,  then either $t = (c - N)/(q-1)$ or $t = (c - N)/(q-1) - 1$,  from which we obtain either $a_N = 1$ or $a_N = q$.  In each case,  we may conclude that \eqref{eq:opt} must have a solution of the desired form,  as any permutation of $(a_1,\dots,  a_N)$ is also a solution.
\end{proof}

\begin{lemma}\label{nonzero points}
Given a nonzero polynomial $f \in \mathsf{F}_q[x_1,\dots,  x_N]$ of degree $d$,  we have 
\[
|V(f)| \le q^N \left( 1  - q^{-\left\lceil \frac{d}{q-1} \right\rceil} \right).
\]
\end{lemma}
\begin{proof}
Since $a^{q}- a=0$ for any $a \in\mathsf{F}_q$,  we may assume $\deg_{x_{i}} f\le q-1$ for each $i \in [N]$.  In particular,  we have $0 \le d \le (q-1)N$.  By Lemma~\ref{alon-furedi bound},  
\[
\lvert \mathsf{F}_q^N \setminus V(f) \rvert \ge \alpha \coloneqq \min  \left\lbrace
z_1 \cdots z_N: z_i \in [r_i], \; i \in [N],  \;  z_1 + \cdots + z_N = qN - d
\right\rbrace.
\] 
According to Lemma~\ref{adjustment},  
\[
\alpha = q^{ \lfloor N - \frac{ d}{q-1} \rfloor } \left( (q-1)N - d + 1 - (q-1) \left\lfloor N - \frac{d}{q-1} \right\rfloor \right) \ge q^{N - \lceil \frac{d}{q-1} \rceil}.\qedhere
\]
\end{proof}

Let $\mathbb{V}_{1},\dots,  \mathbb{V}_{d}$ be vector spaces over a field $\mathsf{K}$.  Suppose $R^{(i)}$ is the ring of polynomials on $\mathbb{V}_i$ for each $i \in [d]$.  Denote $R \coloneqq R^{(1)} \otimes \cdots \otimes R^{(d)}$.  We observe that $R$ is $\mathbb{N}^d$-graded: 
\[
R_{m_1,\dots,  m_d} \coloneqq R^{(1)}_{m_1} \otimes \cdots \otimes R^{(d)}_{m_d},\quad (m_1,\dots,  m_d) \in \mathbb{N}^d.
\]
Here $R^{(i)}_m$ is the subspace of  $R^{(i)}$ consisting of degree $m$ homogeneous polynomials for each $ i \in [d]$.  We equip $\mathbb{N}^d$ with the following partial order: 
\begin{equation}\label{eq:order}
(m_1,\dots, m_d) \preceq (l_1,\dots,  l_d) \iff m_i \le l_i \text{~for each~} i \in [d].
\end{equation}
Moreover,  we define 
\begin{equation}\label{eq:sigma}
\sigma: \mathbb{N} \to \mathbb{N},\quad \sigma(m_1,\dots,  m_d) \coloneqq m_1 + \cdots + m_d.
\end{equation}
The proposition that follows is a multilinear analogue of Proposition~\ref{coro1}.
\begin{proposition}[Non-vanishing maximal minor II]\label{lem2}
Let $\mathbb{V}_{1},\dots,  \mathbb{V}_{d}$ be vector spaces over a field $\mathsf{F}_q$ and let $R^{(1)},\dots,  R^{(d)}$ be rings of polynomials on $\mathbb{V}_1,\dots,  \mathbb{V}_d$ over $\mathsf{F}_q$,  respectively.  Denote $R \coloneqq R^{(1)} \otimes \cdots \otimes R^{(d)}$.  Suppose the matrix $M = (M_{k,j})_{(k,j) \in [m+1] \times [n]} \in R^{(m+1) \times n}$ satisfies
\begin{enumerate}[label=(\alph*)]
\item\label{lem2:item1} For each $ k \in [m+1]$,  $M_{k,1},\dots,  M_{k,n}$ are multi-homogeneous of the same multi-degree $\delta_k \in \mathbb{N}^d$.  Moreover,  $\delta_{m+1} = (1,\dots,  1)$ and $\delta_i \prec \delta_{m+1}$ for any $i \in [m]$.
\item\label{lem2:item2} There exists some $v \in \mathbb{V} \coloneqq \mathbb{V}_1 \times \cdots \times \mathbb{V}_d$ such that the first $m$ rows of $M(v) \coloneqq (M_{k,j}(v)) \in \mathsf{F}_q^{(m+1) \times n}$ are linearly independent. 
\item\label{lem2:item3} Let $F: \mathbb{V} \to \mathsf{F}_q^n$ be the $d$-multilinear map defined by $F(v) \coloneqq (M_{m+1,1}(v),\dots,  M_{m+1,n}(v))$.  We have $\AR (F) > d + m + \lceil m(d-1)/(q-1) \rceil$.
\end{enumerate}
Then $\rank M(a) = m+1$ for some $a \in \mathbb{V}$. 
\end{proposition}
\begin{proof}
Denote $N \coloneqq \dim \mathbb{V} = \sum_{i=1}^d \dim \mathbb{V}_i$.  By permuting columns of $M$,  \ref{lem2:item2} allows us to assume that the upper left $m \times m$ submatrix of $M$ is non-singular at some $v\in V$.  Let $\Delta \in R$ be the determinant of this submatrix.  According to \ref{lem2:item1},  $\Delta$ is a nonzero polynomial and the multi-degree of $\Delta$ is at most $\delta_1 + \cdots + \delta_m$ with respect to the partial order $\preceq$.  In particular,  $\deg \Delta \le \sigma (\delta_1 + \cdots + \delta_m) \le m(d-1)$.  Let $D(\Delta) \coloneqq \mathbb{V} \setminus V(\Delta)$.  Lemma~\ref{nonzero points} implies that
\[
\lvert D(\Delta) \rvert \ge q^{N - \left\lceil \frac{m(d-1)}{q-1} \right\rceil}.  
\]

We prove the existence of $b \in D(\Delta)$ such that $M(b)$ has full rank by contradiction.  Suppose  $\rank M(v) \le m$ for each $v\in D(\Delta)$.  The linear system 
\[
\begin{bmatrix}
y_1 & \cdots & y_{m+1} 
\end{bmatrix} M(v) = 0
\]
must have a solution of the form $(c_{v,1},\dots,  c_{v,m},1) \in \mathsf{F}_q^{m + 1}$.  The Pigeonhole principle implies there is a subset $X \subseteq D(\Delta)$ and $(c_1,\dots, c_m) \in \mathsf{F}_q^m$ such that 
\begin{itemize}
\item $(c_1,\dots, c_m,1) M(v) = 0$ for any $v\in X$.
\item $|X| \ge q^{N - \left\lceil \frac{m(d-1)}{q-1} \right\rceil - m}$.
\end{itemize}
By \ref{lem2:item1},  for each $ i \in [m]$,  there is a proper subset $Q_i \subseteq [d]$ such that $M_{i,1},\dots,  M_{i, n} \in \otimes_{s \in Q_i} \mathbb{V}_s^\ast$.  Denote $P_i \coloneqq [d] \setminus Q_i$ and $\mathbb{V}'_i \coloneqq \mathbb{V}_i \oplus \mathbb{F}_q$ for each $i \in [m]$.  We consider a $d$-multilinear map $F': \mathbb{V}'_1 \times \cdots \times \mathbb{V}'_d \to \mathbb{F}^n$ defined by 
\[
F' (v_1 \oplus \lambda_1,\dots,  v_d \oplus \lambda_d) \coloneqq  F(v) + \sum_{i=1}^m c_i \left( \prod_{s \in P_i} \lambda_s \right) \left( M_{i,1}(v),  \dots,  M_{i,n}(v) \right),
\]
where $v \coloneqq (v_1,\dots,  v_d)$.  We observe that $F' (v_1 \oplus 1,\dots,  v_d \oplus 1) = 0$ for any $(v_1,\dots,  v_d) \in X$,  from which we obtain 
\[
\AR(F') \le (N + d) - \log_q |X| \le d + m  + \left\lceil \frac{m(d-1)}{q-1} \right\rceil.
\]
On the other hand,  we have $F'|_{\mathbb{V}} = F$.  This together with \cite[Claim~3.2]{lovett2018analytic} leads to 
\[
\AR(F) \le \AR(F') \le  d + m  + \left\lceil \frac{m(d-1)}{q-1} \right\rceil,
\]
which contradicts to \ref{lem2:item3}. 
\end{proof}
\begin{remark}
The homogeneous polynomials,  the degree and the height in Proposition~\ref{coro1} are respectively replaced by the multi-homogeneous polynomials,  the multi-degree and the analytic rank in Proposition~\ref{lem2}.  This is because the analytic rank for a vector space spanned by multi-linear polynomials is an analogue of the height for an ideal generated by homogeneous polynomials.  Moreover,  Proposition~\ref{lem2}~\ref{lem2:item2} is stronger than Proposition~\ref{coro1}~\ref{lem1:item2},  and the lower bound for the analytic rank in Proposition~\ref{lem2}~\ref{lem2:item3} can be chosen to be independent of $q$,  since $d+h+\left\lceil\frac{h(d-1)}{q-1}\right\rceil\le d+h+h(d-1)=(h+1)d$.  
\end{remark}

In the sequel,  we will also need the following consequence of Serre's inequality of height.
\begin{lemma}[Height is non-increasing under quotient]\label{lem:height}
Let $R$ be a polynomial ring and let $I,J$ be two ideals.  Then we have $\height(J) \ge \height( (I+J)/I )$.
\end{lemma}
\begin{proof}
Denote $h \coloneqq \height( (I+J)/I ) $.  By definition,  there are prime ideals $q_0,\dots,  q_h \subseteq R$ such that $I \subseteq q_0 \subsetneq \cdots \subsetneq q_h $ and $q_h$ is minimal among primes over $I+J$.  Then we have $\height(I + J) \ge \height(q_0) + h \ge \height(I) + h$.  By Serre's inequality of height \cite[Chapter V,  \S B. ~6,  Theorem~3]{serre2000local},  we also have  $\height(I+J)\le \height(I)+ \height(J)$,  which implies $\height(J)\ge h$.
\end{proof}
\section{Bounding isotropy and completeness indices via other invariants}\label{sec:estimate}
This section is devoted to establishing connections between the isotropy and completeness indices and other existing invariants,  including the partition rank,  the geometric rank,  the analytic rank and the height.  Since $\Hom(\mathbb{V}_1 \times \cdots \times \mathbb{V}_d,  \mathbb{W}) \simeq \mathbb{V}_1^\ast \otimes \cdots \mathbb{V}_d^\ast \otimes \mathbb{W}$,  every $d$-multilinear map $F \in \Hom(\mathbb{V}_1 \times \cdots \times \mathbb{V}_d,  \mathbb{W})$ uniquely corresponds to a $(d+1)$-tensor $T_F \in  \mathbb{V}_1^\ast \otimes \cdots \mathbb{V}_d^\ast \otimes \mathbb{W}$.  In the sequel,  we denote 
\begin{equation}\label{eq:rankmap}
\GR(F) \coloneqq \GR(T_F),  \quad  \PR(F) \coloneqq \PR(T_F),\quad \Q(F) \coloneqq \Q(T_F),\quad \AR(F) \coloneqq \AR(T_F).
\end{equation}

We first prove Theorem~\ref{lem-isotropic-general} on the relation between the isotropy index and the partition rank.
\begin{proof}[Proof of Theorem~\ref{lem-isotropic-general}]
Denote $t \coloneqq \PR(F)$ and $n \coloneqq \dim \mathbb{W}$.  By identifying $\mathbb{W}$ with $\mathsf{K}^n$,  we may write $F = (F_1,\dots,  F_n)$, where $F_k \in \Hom(\mathbb{V}_1\times \cdots \times \mathbb{V}_d,  \mathsf{K})$ for each $k \in [n]$.  According to the definition of the partition rank,  for each $j \in [t]$ and $k \in [n]$,  there exist a partition $P_j \sqcup Q_j = [d]$ with $Q_j \ne \emptyset$ and multilinear functions $f_j \in \Hom(\prod_{s \in Q_j} \mathbb{V}_s,  \mathsf{K})$ and $g_{kj} \in \Hom( \prod_{s \in P_j} \mathbb{V}_s,  \mathsf{K})$,  such that $F_k = f_1 g_{k1} + \cdots + f_t g_{kt}$.  Here we naturally identify $f_j,  g_{kj}$ as multilinear polynomials on $\mathbb{V}_1 \times \cdots \times \mathbb{V}_d$.  Although the proofs of \ref{lem-isotropic-general:item1},  \ref{lem-isotropic-general:item2} and \ref{lem-isotropic-general:item3} are all based on the same idea,  each requires a slightly different construction.  Thus,  we provide a detailed proof for each case,  for completeness. 
\begin{itemize}
\item[$\diamond$] We first prove \ref{lem-isotropic-general:item1}.  If $m = 1$ or $m = t$,  the inequality is obvious.  Thus,  we may assume that $2 \le m$ and $t + 1 \le m$.  Suppose $r$ is the minimal positive integer such that $m \le t (r+1)^{d-1}+r+1$.  Since $2 \le m$ and $1 \le t $,  we must have
\[
1 \le r \le tr^{d-1}+r < m < t.
\]

We claim that for each $i \in [d]$,  there are linearly independent vectors $v_{i,1},\dots,  v_{i,r} \in \mathbb{V}_i$ such that 
\[
f_j(v_{1,l_1},\dots,  v_{d,  l_d}) = 0,\quad j \in [t],  \; l_1,\dots,  l_d \in [r].
\]
If the claim is true,  then we have 
\[
F_k(v_{1,l_{1}},\cdots,v_{d,l_{d}})=\sum_{j=1}^{t}f_{j}(v_{1,l_{1}},\cdots,v_{d,l_{d}})g_{kj}(v_{1,l_{1}},\cdots,v_{d,l_{d}})=0
\]
for any $k \in [n],  l_1,\dots,  l_d \in [r]$.  This implies $r \le \alpha(F)$,  and $m \le t \alpha(F)^{d-1} + \alpha(F) + 1$. 

It is left to prove the claim.  We construct $X_s \coloneqq \{ v_{i,l} \in \mathbb{V}_i: (i,l) \in [d] \times [s]\}$ inductively on $s$,  such that 
\begin{enumerate}[label=(a\arabic*)]
\item\label{lem-isotropic-general:item i} $v_{i,1},\dotsm  v_{i,s}$ are linearly independent for each $ i \in [d]$.
\item\label{lem-isotropic-general:item ii} $f_j(v_{1,l_1},\dots,  v_{d,  l_d}) = 0$ for any $ j \in [t]$ and $l_1,\dots,  l_d \in [s]$.
\end{enumerate}
For $s = 1$,  we arbitrarily choose nonzero $v_{i,1} \in \mathbb{V}_{i}$ for each $ i \in [d-1]$,  and choose $v_{d,1}$ to be a nonzero solution of the homogeneous linear system:
\begin{equation}\label{lem-isotropic-general:eq2}
f_j(v_{1,1},\dots,  v_{d-1,1},  z) = 0,\quad j \in [t].
\end{equation}
There are $t$ linear constraints on $\dim \mathbb{V}_d$ variables in \eqref{lem-isotropic-general:eq2}.  A desired $v_{d,1}$ must exist,  since $t \le m - 1 < \dim \mathbb{V}_d$.

Assume that we already have $X_s = \{v_{i,l} \in \mathbb{V}_: (i,l) \in [d] \times [s]\}$ for some $s \in [r-1]$.  We want to find $v_{i,s+1} \in \mathbb{V}_i$ for each $ i \in [d]$ such that $X_{s + 1} = X_s \sqcup \{v_{i,s+1}:  i \in [d]\}$ satisfies \ref{lem-isotropic-general:item i} and \ref{lem-isotropic-general:item ii}.  To this end,  first we notice that $\dim \mathbb{U}_1 \ge n - t s^{d-1} > s$,  where
\[
\mathbb{U}_1 \coloneqq \Bigl\{ z \in \mathbb{V}_1 :
f_j(z,  v_{2,l_2},\dots,v_{d,l_d}) = 0,  \; l_2,\dots,\dots,l_d \in [s],\; j \in [t] \Bigr\}.
\]
Thus,  we may pick $v_{1,  s+1} \in \mathbb{U}_1 \setminus \spa \{v_{1,1},\dots,  v_{1,s}\}$.  Next,  for each $2 \le i \le d$,  we inductively consider 
\begin{equation*}
\begin{split}
\mathbb{U}_i \coloneqq \Bigl\{ z \in \mathbb{V}_i :
&f_j(v_{1,l_1},\dots,  v_{i-1,l_{i-1}},  z,   v_{i,l_i},\dots,   v_{d,l_d}) = 0, \\
& l_1,\dots, l_{i-1} \in [s+1],  \; l_{i+1},  \dots,l_d \in [s],\; j \in [t] \Bigr\}.
\end{split}
\end{equation*}
Since $\dim \mathbb{U}_i \ge m - t (s+1)^{i-1}s^{d - i}> m - t(s+1)^{d-1}  >  s$,  we may pick $v_{i,s+1} \in \mathbb{U}_i \setminus \spa \{v_{i,1},\dots,  v_{i,s}\}$.  It is clear that $v_{1,s+1},\dots,  v_{d,  s+1}$ satisfy the requirement,  and the induction is complete.
\item[$\diamond$] Next,  we prove \ref{lem-isotropic-general:item2}.  Suppose $\mathbb{V}_1 = \cdots = \mathbb{V}_d = \mathbb{V}$.  Let $r$ be the minimal positive integer such that $m \le t\binom{r}{d-1}+ r$.  Then we have $t\binom{r-1 }{d-1} + r -1 < m$.  We claim that there exist $v_1,\dots,  v_r \in \mathbb{V}$ such that 
\begin{enumerate}[label=(b\arabic*)]
\item\label{lem-isotropic-general:item i-1} $v_1,\dots,  v_r$ are linearly independent.  
\item\label{lem-isotropic-general:item i-2} $f_j(v_{l_1},\dots,  v_{l_d}) = 0$ for any $j \in [t]$ and $1 \le l_1 < \cdots < l_d \le r$.
\end{enumerate}
If the claim is true,  then by the same argument as in the proof of \ref{lem-isotropic-general:item1},  we obtain $r \le \alpha_{\w}(F)$ and the inequality follows immediately.

We construct $v_1,\dots,  v_r$ inductively.  Let $v_1 \in \mathbb{V}$ be any nonzero element.   Suppose that for some $s \in [r-1]$,  we have linearly independent $v_1 ,\dots,  v_s \in \mathbb{V}$ such that $f_j(v_{l_1},  \dots,  v_{l_d}) = 0$,  $j \in [t],  1 \le l_1 < \cdots <l_d \le r$.  We consider the vector space
\begin{equation}\label{lem-isotropic-general:eq3}
\mathbb{U}_s \coloneqq \left\lbrace z \in \mathbb{V}: f_j(v_{l_1}, \dots,  v_{l_{d-1}},  z) = 0,\; j \in [t],\; 1 \le l_1 < \cdots < l_{d-1} \le s
\right\rbrace.
\end{equation}
Since $\dim \mathbb{U}_s \ge m - t \binom{s}{d-1} > r-1 \ge s$,  we may pick $v_{s+1} \in \mathbb{U}_s \setminus \spa\{v_1,\dots,  v_s\}$.  It is straightforward to verify that $v_1,\dots,   v_r$ satisfy \ref{lem-isotropic-general:item i-1} and \ref{lem-isotropic-general:item i-2}.
\item[$\diamond$] Lastly,  we prove \ref{lem-isotropic-general:item3}.  Suppose $\mathbb{V}_1 = \cdots = \mathbb{V}_d = \mathbb{V}$.  The inequality trivially holds if $m = t$.  Hence we may assume $t < m$ in the rest of the proof.   Let $r$ be the minimal positive integer such that $m \le t \binom{r + d - 1}{d-1} + r$.  Then we have $t \binom{r + d - 2}{d-1} + r - 1 < m$.  Similar to the proof of \ref{lem-isotropic-general:item2},  it is sufficient to prove the existence of $v_1,\dots,  v_r \in \mathbb{V}$ such that 
\begin{enumerate}[label = (c\arabic*)]
\item $v_1,\dots,  v_r$ are linearly independent. 
\item $f_j(v_{l_1},\dots,  v_{l_d}) = 0$ for any $j \in [t]$ and $1 \le l_1 \le \cdots \le l_d \le r$.
\end{enumerate} 

Again,  we construct $v_1,\dots,  v_r$ inductively.  Since $t < n$ and $\mathsf{K}$ is algebraically closed,  the system of homogeneous polynomials  
\[
f_1(z,\dots,  z) = \cdots = f_t (z,\dots,  z) = 0
\]
has a nonzero solution $z = v_1\in \mathbb{V}$.  Suppose that for some $s \in [r-1]$,  we have linearly independent $v_1 ,\dots,  v_s \in \mathbb{V}$ such that $f_j(v_{l_1},  \dots,  v_{l_d}) = 0$,  $j \in [t],  1 \le l_1 \le \cdots \le l_d \le r$.  We consider the variety 
\begin{align*}
V_s \coloneqq \Bigl\{
z \in \mathbb{V} :
&f_j(v_{l_1}, \dots, v_{l_u}, \underbrace{z, \dots, z}_{\text{$(d-u)$ times}}) = 0,  \\
& j \in [t],\; 0 \le u \le d-1,\; 1 \le l_1 \le \dots \le l_u \le s
\Bigr\}.
\end{align*}
We notice that 
\[
\dim V_s \ge m - t\sum_{u=0}^{d-1} \binom{s + u - 1}{u} = m - t\binom{s + d - 1}{d- 1} \ge m - t\binom{r + d -2}{d-1}> s.
\]
Since $\mathsf{K}$ is algebraically closed, there must exist some nonzero $v_{s+1} \in V_s \setminus \spa\{v_1,\dots,  v_{s}\}$,  and this completes the proof.   \qedhere
\end{itemize}
\end{proof}
\begin{remark}
By Theorem~\ref{lem-isotropic-general},  $\dim \mathbb{V} / \PR(F)$ provides a lower bound for $\alpha_{\w}(F)$ for any $F \in \Alt^d(\mathbb{V},\mathbb{W})$.  A natural question is whether $\alpha_{\w}(F)$ can be bounded from above in terms of $\dim \mathbb{V} / \PR(F)$.  Yet, the following example shows that this is not the case.  Let $\mathcal{G} = (V,E)$ be a graph and let $T_{\mathcal{G}} \in \Alt^2(\mathsf{K}^n,\mathsf{K}^m)$ be its Tutte matrix \cite[page~109]{Tutte47},  where $n \coloneqq |V|$ and $m \coloneqq |E|$.  According to \cite[Proposition~5.1]{qiao2020tur},  we have $\alpha_{\w}( T_{\mathcal{G}} ) = \alpha(\mathcal{G})$,  the independence number of $G$.  Moreover,  we recall from \cite[Corollary~3.3]{herzog2010binomial} that 
\[
\GR( T_{\mathcal{G}} ) = \min \{ 
2n - |S| - c(\mathcal{G}(S)): S\subseteq V
\},
\]
where $\mathcal{G}(S)$ is the induced subgraph of $\mathcal{G}$ on $S$ and $c(\mathcal{H})$ denotes the number of connected components of a graph $\mathcal{H}$.  Thus,  when $\mathcal{G}$ is the complete bipartite graph $\mathcal{K}_{t,t}$ with parts of size $t$,  we have $n = 2t$,  $\alpha_{\w}(T_{\mathcal{K}_{t,t}}) = \alpha(\mathcal{K}_{t,t}) = t$ and $\GR(T_{\mathcal{K}_{t,t}}) = 2 t - 1$.  By \cite[Theorem~5]{kopparty2020geometric},  we have $\GR( T_{\mathcal{K}_{t,t}} ) \le \PR( T_{\mathcal{K}_{t,t}} )$.  Since $t$ can be arbitrarily large whereas $n/\PR( T_{\mathcal{K}_{t,t}} ) \le n/\GR( T_{\mathcal{K}_{t,t}} )  \le 2$,  we may conclude that there is no function $f: \mathbb{N} \to \mathbb{N}$ such that $\alpha_{\w}( T_{\mathcal{K}_{t,t}} ) \le f (n/\PR( T_{\mathcal{K}_{t,t}} ))$ for any $t$.
\end{remark}

In what follows,  we establish Theorem~\ref{lem-complete1} which provides lower bounds on the completeness index in terms of the geometric rank,  the analytic rank and the height,  respectively.  In contrast to Theorem~\ref{lem-isotropic-general} which states that a small partition rank forces a large isotropic index,  we show that the larger the geometric rank,  the analytic rank or the height,  the larger the completeness index.  Before we proceed,  we recall from \eqref{eq:polar} and \eqref{eq:LF} that,  when $\ch(\mathsf{K})=0$ or $\ch(\mathsf{K}) > d$,  every $F \in \Sym^d(\mathbb{V},\mathbb{W})$ determines a subspace $\mathbb{L}_F \subseteq \mathsf{K}[x_1,\dots,x_m]_d$ with $m = \dim \mathbb{V}$.  We denote by $\mathfrak{a}_F$ the ideal generated by $\mathbb{L}_F$ in $\mathsf{K}[x_1,\dots,x_m]$.  
\begin{proof}[Proof of Theorem~\ref{lem-complete1}]
For simplicity,  we set
\[
L_s \coloneqq s^d,\quad N_s \coloneqq \binom{s}{d},\quad s \in \mathbb{N}.
\]
We identify $\mathbb{W}$ with $\mathsf{K}^n$ so that we can write each $F \in \Hom(\mathbb{V}^d,  \mathbb{W})$ as $F = (F_1,\dots,  F_n)$ for some $F_1,\dots,  F_n \in \Hom(\mathbb{V}^d,  \mathsf{K})$.  Denote $R \coloneqq \mathsf{K}[x_{1,1},\dots,   x_{d,m}]$,  the polynomial ring over $\mathsf{K}$ with $dm$ variables $x_{1,1},\dots,  x_{d,m}$.  For each $i \in [d]$,  we write $w_i \coloneqq (x_{i,1},\dots,  x_{i,m})$.  
\begin{itemize}[listparindent=1em]
\item[$\diamond$] We first prove \ref{lem-complete1:gen}.  For each $s \in \mathbb{N}$,  we define 
\begin{equation}\label{lem-complete1:gen:eq-1}
\mathcal{L}_s \coloneqq \left\lbrace
(j_1,\dots,  j_d) \in \mathbb{N}^d: j_1,\dots  j_d \in [s]
\right\rbrace.
\end{equation}
We denote the $i$-th element of $\mathbf{L} \in \mathcal{L}_s$ by $\mathbf{L}(i)$ for each $ i \in [d]$,  and we arbitrarily order elements in $\mathcal{L}_s$ as:
\begin{equation}\label{lem-complete1:gen:eq0}
\mathbf{L}_1 <  \dots < \mathbf{L}_{L_s}.
\end{equation}
\begin{itemize}[listparindent=1em]
\item[$\vartriangleright$] Assume that $\mathsf{K}$ is an infinite field.  Denote $t \coloneqq \GR(F)$.  Suppose that $r$ is the maximal integer such that $L_r \le t$.  Then we have $L_{r+1} > t$.  We claim that there exist $v_{i,j}\in \mathbb{V}_i,  (i,j) \in [d]\times [r]$ such that 
\[ 
F(v_{1,j_1},\dots,  v_{d,j_d})\in \mathbb{W},\quad (j_1,\dots,  j_d) \in \mathcal{L}_r
\]
are linearly independent.  This implies that $\beta (F) \ge r$ and the inequality follows immediately. 
 
To prove the claim,  we show inductively that for each $k \in [L_r]$,  there exist $v_{k,i,j} \in \mathbb{V}_i,  (i,j) \in [d] \times [r]$ such that 
\[ 
F(v_{k, 1,\mathbf{L}_1(1)},\dots,  v_{k, d,\mathbf{L}_1(d)}),\dots,  F(v_{k, 1,\mathbf{L}_k(1)},\dots,  v_{k, d,\mathbf{L}_k(d)}) \in \mathbb{W}
\]
are linearly independent.  Then we obtain the claimed vectors by taking $v_{i,j} = v_{L_r, i,  j}$ for $1 \le i \le d,  1 \le j \le r$.  

If $k = 1$,  there exist $u_1\in \mathbb{V}_1,\dots,  u_d \in \mathbb{V}_d$ such that $F(u_1,\dots,  u_d) \ne 0$ since $F$ is nonzero.  Thus,  we may extend $u_1,\dots,  u_d$ and reindex them to obtain $\{ v_{k,i,j} \in \mathbb{V}_i:  (i,j) \in [d] \times [r]\}$ such that 
\[
F(v_{1,1,\mathbf{L}_1(1)},  \dots,  v_{1,d,\mathbf{L}_1(d)}) = F(u_1,\dots,  u_d) \ne 0.
\]

Suppose that we already have $\{ v_{k,i,j} \in \mathbb{V}_i: (i,j) \in [d] \times [r]\}$ for some $k \in [L_r - 1]$.  We consider the matrix $M \in R^{(k+1) \times n}$ defined as 
\begin{equation}\label{lem-complete1:gen:eq1}
M(w_1,\dots,  w_d) \coloneqq \begin{bmatrix}
F_1(w'_{1,1},\dots,  w'_{1,d}) & \cdots & F_n(w'_{1,1},\dots,  w'_{1,d}) \\
\vdots & \ddots & \vdots  \\
F_1(w'_{k,1},\dots,  w'_{k,d}) & \cdots & F_n(w'_{k,1},\dots,  w'_{k,d}) \\
F_1(w_1,\dots,  w_d) & \cdots &  F_n(w_1,\dots,  w_d)
\end{bmatrix},
\end{equation}
where for each $a \in [k]$ and $b \in [d]$,  
\begin{equation}\label{lem-complete1:gen:eq2}
w'_{a,b} \coloneqq 
\begin{cases}
w_{b} \quad &\text{if~} \mathbf{L}_a(b) = \mathbf{L}_{k+1}(b)  \\ 
v_{k,  b,  \mathbf{L}_a(b)} \quad &\text{otherwise}
\end{cases}.
\end{equation}
Clearly,  the first $k$ rows of $M(v_{k,1,\mathbf{L}_{k+1}(1)},\dots,  v_{k,d, \mathbf{L}_{k+1}(d)})$ are linearly independent by the induction hypothesis,  and elements in the $s$-th row of $M$ are homogeneous of the same degree $d_s$ for each $s \in [k+1]$.

Since $\mathbf{L}_a \ne \mathbf{L}_{k+1}$ for any $a \in [k]$,  we derive that $d_s < d_{k+1}$ if $s \in [k]$.  Moreover,  we have 
\[
\height \mathfrak{a} = \GR(F) = t > L_r - 1 \ge k,
\] 
where $\mathfrak{a}$ is the ideal of $R$ generated by the last row of $M$.  By Proposition~\ref{coro1},  there exist $u_1\in \mathbb{V}_1,\dots,  u_d \in \mathbb{V}_d$ such that $\rank M(u_1,\dots,  u_d) = k+1$.

We define for each $(i,j) \in [d] \times [r]$ that 
\begin{equation}\label{lem-complete1:gen:eq3}
v_{k+1, i,  j} \coloneqq \begin{cases}
u_j \quad &\text{if~} j = \mathbf{L}_{k+1}(i) \\
v_{k,i,  j} \quad &\text{otherwise} 
\end{cases}.
\end{equation}
Then we have 
\[
\begin{bmatrix}
F(v_{k+1,1,\mathbf{L}_1(1)},\dots,  v_{k+1,d,\mathbf{L}_1(d)}) \\
\vdots \\
F(v_{k+1,1,\mathbf{L}_k(1)},\dots,  v_{k+1,d,\mathbf{L}_k(d)})
\end{bmatrix} = M(u_1,\dots,  u_d)
\]
and this completes the proof.

\item[$\vartriangleright$] Suppose $\lvert \mathsf{K} \rvert = q$.  In this case,  we denote $t \coloneqq \AR(F)$ and let $r$ be the maximal integer such that $d+N_r +\lceil L_r (d-1)/(q-1) \rceil \le t$. This implies $t < d+L_{r+1} +\lceil L_{r+1} (d-1)/(q-1) \rceil$.  We claim that there is a subset $\{v_{i,j} \in \mathbb{V}_i:  (i,j) \in [d] \times [r] \}$ such that 
\[
F(v_{1,j_1},\dots,  v_{d,j_d}) \in \mathbb{W},\quad (j_1,\dots,  j_d) \in \mathcal{L}_r
\]
are linearly independent,  so that $r \le \beta (F)$ and the desired inequality is obtained.  

The rest of the proof is similar to that for the case $|\mathsf{K}| = \infty$,  so we will only provide a sketch.  We want to show that for each $1 \le k \le L_r$,  there exist $\{ v_{k,i,j} \in \mathbb{V}_i: (i,j) \in [d] \times [r] \}$ such that 
\[ 
F(v_{k, 1,\mathbf{L}_1(1)},\dots,  v_{k, d,\mathbf{L}_1(d)}),\dots,  F(v_{k, 1,\mathbf{L}_k(1)},\dots,  v_{k, d,\mathbf{L}_k(d)}) \in \mathbb{W}
\]
are linearly independent.  This is obvious for $k = 1$ since $F$ is nonzero.  We assume that $\{ v_{k,i,j} \in \mathbb{V}_i: (i,j) \in [d] \times [r] \}$ is given for some $k \in [L_r - 1]$.  We observe that the algebra $R$ is $\mathbb{N}^d$-graded since $R \simeq  S^{\otimes d}$,  where $S$ is the polynomial ring over $\mathsf{K}$ with $m$ variables.  In particular,  $F_1,\dots,  F_n$ as elements of $R$,  have multi-degree $(1,\dots, 1) \in \mathbb{N}^d$.  Moreover,  we recall that $\mathbb{N}^d$ is equipped with the partial order $\preceq$ defined by \eqref{eq:order}.  Let $M \in R^{ (k+1) \times n }$ be the matrix defined by \eqref{lem-complete1:eq1} and \eqref{lem-complete1:eq2}.  It is straightforward to verify that $M$ satisfies \ref{lem2:item1}--\ref{lem2:item3} in Proposition~\ref{lem2}.  Thus,  there exist $u_1 \in \mathbb{V}_1,\dots,  u_d \in \mathbb{V}_d$ such that $\rank M(u_1,\dots,  u_d) = k+1$.  The induction is complete if we define $v_{k+1,i,j}$ as in \eqref{lem-complete1:gen:eq3}. 
\end{itemize}

\item[$\diamond$] The proof of \ref{lem-complete1:alt} shares the same idea as that of \ref{lem-complete1:gen},  but differs slightly because of the symmetries of alternating multilinear forms.  For each $s \in \mathbb{N}$,  we define 
\[
\mathcal{J}_s \coloneqq \left\lbrace
(j_1,\dots,  j_d) \in \mathbb{N}^d: 1 \le j_1 < \cdots < j_d \le s
\right\rbrace.
\]
We denote the $i$-th element of $\mathbf{J} \in \mathcal{J}_s$ by $\mathbf{J}(i)$ for each $ i \in [d]$,  and we arbitrarily order elements in $\mathcal{J}_s$ as:
\[
\mathbf{J}_1 <  \dots < \mathbf{J}_{N_s}.
\]
\begin{itemize}[listparindent=1em]
\item[$\vartriangleright$] Assume that $\mathsf{K}$ is an infinite field.  Denote $t \coloneqq \GR(F)$.  Suppose that $r$ is the maximal integer such that $N_r \le t$.  Then we have $N_{r+1} > t$.  We claim that there exist $v_1,\dots,    v_r \in \mathbb{V}$ such that $F(v_{j_1},\dots,  v_{j_d})\in \mathbb{W}, (j_1,\dots,  j_d) \in \mathcal{J}_r$ are linearly independent.  This implies that $\beta_{\w} (F) \ge r$ and the inequality follows immediately. 

To prove the claim,  we show inductively that for each $k \in [N_r]$,  there are $v_{k,1},\dots,  v_{k,r}$ in $\mathbb{V}$ such that $F(v_{k, \mathbf{J}_1(1)},\dots,  v_{k, \mathbf{J}_1(d)}),\dots,  F(v_{k, \mathbf{J}_k(1)},\dots,  v_{k, \mathbf{J}_k(d)}) \in \mathbb{W}$ are linearly independent.  Then we obtain the claimed vectors by taking $v_j = v_{N_r, j}$ for $j \in [r]$.  

If $k = 1$,  there exist $u_1,\dots,  u_d \in \mathbb{V}$ such that $F(u_1,\dots,  u_d) \ne 0$ since $F$ is nonzero.  Thus,  we may extend $u_1,\dots,  u_d$ and reorder them to obtain $v_{1,1},\dots,  v_{1,d}$ such that 
\[
F(v_{1,\mathbf{J}_1(1)},  \dots,  v_{1,\mathbf{J}_1(d)}) = F(u_1,\dots,  u_d) \ne 0.
\]

Suppose that we already have $v_{k,1},\dots,  v_{k,r} \in \mathbb{V}$ for some $k \in [N_r - 1]$.  We consider the matrix $M \in R^{(k+1) \times n}$ defined as 
\begin{equation}\label{lem-complete1:eq1}
M(w_1,\dots,  w_d) \coloneqq \begin{bmatrix}
F_1(w'_{1,1},\dots,  w'_{1,d}) & \cdots & F_n(w'_{1,1},\dots,  w'_{1,d}) \\
\vdots & \ddots & \vdots  \\
F_1(w'_{k,1},\dots,  w'_{k,d}) & \cdots & F_n(w'_{k,1},\dots,  w'_{k,d}) \\
F_1(w_1,\dots,  w_d) & \cdots &  F_n(w_1,\dots,  w_d)
\end{bmatrix},
\end{equation}
where for each $a \in [k]$ and $b \in [d]$,  
\begin{equation}\label{lem-complete1:eq2}
w'_{a,b} \coloneqq 
\begin{cases}
w_{c} \quad &\text{if~} \mathbf{J}_a(b) = \mathbf{J}_{k+1}(c) \text{~for some~} c \in [d] \\ 
v_{k,  \mathbf{J}_a(b)} \quad &\text{otherwise}
\end{cases}.
\end{equation}
Clearly,  the first $k$ rows of $M(v_{k,\mathbf{J}_{k+1}(1)},\dots,  v_{k,\mathbf{J}_{k+1}(d)})$ are linearly independent by the induction hypothesis,  and elements in the $s$-th row of $M$ are homogeneous of the same degree $d_s$,  $s \in [k+1]$.  Since elements in $\mathbf{J}_a$ and $\mathbf{J}_{k+1}$ can not be exactly the same for any $a \in [k]$,  we derive that $d_s < d_{k+1}$ if $s \in [k]$.  Moreover,  we have 
\[
\height \mathfrak{a} = \GR(F) = t > N_r - 1 \ge k,
\] 
where $\mathfrak{a}$ is the ideal of $R$ generated by the last row of $M$.  By Proposition~\ref{coro1},  there exists $u_1,\dots,  u_d \in \mathsf{K}^m$ such that $\rank M(u_1,\dots,  u_d) = k+1$.

We define for each $1 \le j \le r$ that 
\begin{equation}\label{lem-complete1:eq3}
v_{k+1,  j} \coloneqq \begin{cases}
u_c \quad &\text{if~} j = \mathbf{J}_{k+1}(c) \text{~for some~} c \in [d] \\
v_{k,j} \quad &\text{otherwise} 
\end{cases},
\end{equation}
and this completes the proof.

\item[$\vartriangleright$] Suppose $\lvert \mathsf{K} \rvert = q$.  In this case,  we denote $t \coloneqq \AR(F)$ and let $r$ be the maximal integer such that $d+N_r +\lceil N_r (d-1)/(q-1) \rceil \le t$. This implies $t < d+N_{r+1} +\lceil N_{r+1} (d-1)/(q-1) \rceil$.  We claim that there are $v_1,\dots,  v_r \in \mathbb{V}$ such that 
\[
F(v_{j_1},\dots,  v_{j_d}) \in \mathbb{W},\quad (j_1,\dots,  j_d) \in \mathcal{J}_r
\]
are linearly independent,  so that $r \le \beta_{\w} (F)$ and the desired inequality is obtained.  The rest of the proof is exactly the same as that for \ref{lem-complete1:gen},  thus we omit details.  
\end{itemize}

\item[$\diamond$] The proof for \ref{lem-complete1:sym} is a simple modification of that for \ref{lem-complete1:alt}.  Suppose either $\ch(\mathsf{K})=0$ or $\ch(\mathsf{K}) > d$ and $|\mathsf{K}| >  \height(\mathfrak{a}_F)$.  By Proposition~\ref{lem:poly=sym} and Lemma~\ref{lem:height},  we have $\beta_{\s} (F) = \beta_{\p}(L_{\mathbb{F}}) \le \height (\mathfrak{a}_F)$.  To prove the other inequality,  we consider for each $s \in \mathbb{N}$ the set
\[
\mathcal{S}_s \coloneqq \left\lbrace
(j_1,\dots,  j_d) \in \mathbb{N}^d: 1 \le j_1 \le \cdots \le j_d \le s
\right\rbrace,
\]
and we denote the $i$-th element of $\mathbf{S} \in \mathcal{S}_s$ by $\mathbf{S}(i)$ for each $i \in [d]$,  and we arbitrarily order elements in $\mathcal{S}_s$ as:
\[
\mathbf{S}_1 <  \dots < \mathbf{S}_{N_{s + d-1}}.
\]
Similar to the proof of \ref{lem-complete1:alt},  we show that for each $k \in [N_{r + d - 1}]$,  there are $v_{k,1},\dots,  v_{k,r} \in \mathbb{V}$ such that 
\[
F(v_{k, \mathbf{S}_1(1)},\dots,  v_{k,\mathbf{S}_1(d)}),\dots,  F(v_{k, \mathbf{S}_k(1)},\dots,  v_{k,\mathbf{S}_k(d)}) \in \mathbb{W}
\] 
are linearly independent,  where $r$ is the minimal positive integer satisfying $ \height (\mathfrak{a}_F) < N_{r + d}$.  We obtain the inequality when $k = N_{r + d - 1}$.  

The proof for $k = 1$ is trivial as $F$ is nonzero.  Suppose we already have $v_{k,1},\dots,  v_{k,r}$ for some $k \in [N_{r + d - 1}-1]$.  We construct $v_{k+1,1},\dots,  v_{k+1,r}$ by considering the matrix $M = (M_{a,b}) \in R^{(k+1) \times m}$ defined in a similar way as in \eqref{lem-complete1:eq1} and \eqref{lem-complete1:eq2}.  By construction,  we have $\deg_{x_{i,j}} M_{a,b} \le 1$ for any $(i,j,a,b) \in [d] \times [m] \times [k+1] \times [m]$.  This implies 
\[
|\mathsf{K}| > \height(\mathfrak{a}_F) \ge N_{r + d - 1} \ge k + 1.
\]
Moreover,  we notice that $\height (\mathfrak{a}) \ge \height(\mathfrak{a}_F) >  k$.  Here $\mathfrak{a}$ is the ideal of $R$ generated by $F_1,\dots,  F_n$,  and the first inequality follows from $q(\mathfrak{a}) = \mathfrak{a}_F$ where $q: R \to  \mathsf{K}[x_1,\dots,  x_m]$ is the map induced by 
$q(x_{i,j}) = x_j$ for $(i,j) \in [d] \times [m]$.  Thus,  conditions in Proposition~\ref{coro1} are satisfied by $M$ and $\mathsf{K}$,  and this completes the induction step.  

Lastly,  suppose that $\mathsf{K}$ is an infinite field.  If we let $r$ be the minimal positive integer such that $\GR(F) = \height (\mathfrak{a}) < N_{r+d}$ and repeat the above argument,  then we obtain the ``moreover" part.   \qedhere
\end{itemize}
\end{proof} 

\section{Applications}\label{sec:applications}
We consider three applications of Theorems~\ref{lem-isotropic-general} and \ref{lem-complete1}.  First,  we compare the tensor ranks defined in Subsection~\ref{subsec:prelim-rank},  each of which characterizes a salient structural feature of tensors.  For the second application,  we establish a Ramsey-type result for multilinear maps,  and discuss its implications for linear algebra,  group theory and algebraic geometry.  The last application focuses on computational aspects of completeness indices, tensor ranks and height of polynomial ideals. 
\subsection{Tensor ranks}\label{subsec:rank}
Theorem~\ref{lem-complete1} shows that the geometric rank is bounded above by a power of the completeness index.  In what follows,  we relate the subrank to the geometric rank through the completeness index. 
\begin{proposition}[Completeness index vs.  tensor ranks]\label{main theorem 2} 
Let $\mathsf{K}$ be a field and let $\mathbb{V}_1,\dots,  \mathbb{V}_d$ be vector spaces over $\mathsf{K}$.  Given $F \in \Hom(\mathbb{V}_1 \times \cdots \times \mathbb{V}_d,  \mathbb{W})$,  we have 
\begin{equation}\label{main theorem 2:eq1}
\beta(F)  \le \Q(F) \le \GR(F)  < (\beta(F)+1)^d
\end{equation}
if $|\mathsf{K}| = \infty$,  and 
\begin{equation}\label{main theorem 2:eq2}
\beta(F) \le \Q(F) \le \GR(F) \asymp_d \AR(F)  <  d + (\beta(F) + 1)^d  +  \left\lceil\frac{(\beta(F)+1)^{d}(d-1)}{q-1} \right\rceil
\end{equation}
if $|\mathsf{K}| = q$.  Moreover,  for $\mathbb{V}_1 = \cdots = \mathbb{V}_d = \mathbb{V}$ and $F \in \Alt^d(\mathbb{V},  \mathbb{W})$ (resp.  $F \in \Sym^d(\mathbb{V},  \mathbb{W})$),  the inequalities in \eqref{main theorem 2:eq1} and \eqref{main theorem 2:eq2} remain valid if we replace the rightmost terms by the upper bounds in Theorem~\ref{lem-complete1}--\ref{lem-complete1:alt} (resp.  Theorem~\ref{lem-complete1}--\ref{lem-complete1:sym}). 
\end{proposition}
\begin{proof}
The inequality $\Q(F)\le \GR(F)$ follows from \cite[Theorem~1]{kopparty2020geometric}.  By \cite[Proposition~4.4]{chen2024stability} and \cite[Theorem~3]{moshkovitz2024uniform},  we have $\GR(F) \asymp \AR(F)$ in which constants only depend on $d$.  According to Theorem~\ref{lem-complete1}--\ref{lem-complete1:gen},  we have $\GR(F)< (\beta(F)+1)^{d}$ if $|\mathsf{K}| = \infty$,  and $\AR(F)< d + (\beta(F) + 1)^d  +  \left\lceil (\beta(F)+1)^{d}(d-1)/(q-1) \right\rceil$ if $|\mathsf{K}| = q$.  Therefore,  it is left to prove $r \coloneqq \beta(F)\le \Q(F)$.  

Let $\mathbb{U} \coloneqq \mathbb{U}_1 \times \cdots \times \mathbb{U}_d$ be a complete subspace of $F$ with $\dim \mathbb{U}_1 = \cdots = \dim \mathbb{U}_d = r$. Suppose that $I_{r,d}$ is the map defined as in \eqref{prop:ex:item1:eq1}.  As an element in $\Hom(\mathbb{U}_1 \otimes \cdots \otimes \mathbb{U}_d,  \mathbb{W})$,  the restriction map $F|_{\mathbb{U}}$ is injective.  Thus,  we must have $I_{r,d} \unlhd F|_{\mathbb{U}} \unlhd F$,  which implies $r \le \Q(F)$.
\end{proof}
To proceed,  we need the following basic fact from algebraic geometry,  for which we supply a proof in the absence of an appropriate reference.  
\begin{lemma}\label{lem:density}
Let $\mathsf{K}$ be an infinite field.  Then $\mathsf{K}^n$ is Zariski dense in $\overline{\mathsf{K}}^n$. 
\end{lemma} 
\begin{proof}
Let $X$ be the Zariski closure of $\mathsf{K}^n$ in $\overline{\mathsf{K}}^n$.  Suppose $X \ne \overline{\mathsf{K}}^n$.  Then there is some nonzero polynomial $f \in \overline{\mathsf{K}}[x_1,\dots,  x_n]$ such that $f(a) = 0$ for any $a \in X$.  Since coefficients of $f$ are algebraic over $\mathsf{K}$,  we may write $f = \sum_{i=1}^s \alpha_i f_i$,  where $\alpha_1,\dots,  \alpha_s \in \overline{\mathsf{K}}$ are linearly independent over $\mathsf{K}$ and $f_1,\dots,  f_s \in \mathsf{K}[x_1,\dots,  x_n]$.  For each $a \in \mathsf{K}^n$,  we have $0 = f(a) = \sum_{i=1}^s \alpha_i f_i(a)$.  This implies $f_i(a) = 0$ for each $i \in [s]$. Now that $|\mathsf{K}| = \infty$,  we must have $f_1= \cdots = f_s = 0$ \cite[page~228]{Cohn03},  which contradicts to the fact that $f$ is nonzero. 
\end{proof}
Now we are ready to prove Theorem~\ref{cor:GvsS}.
\begin{proof}[Proof of Theorem~\ref{cor:GvsS}]
The relation $\Q(T) = \Omega_k(\GR(T)^{1/(k-1)})$ is a direct consequence of Proposition~\ref{main theorem 2}.  It is sufficient to prove the existence of $T_0$.  Denote $n \coloneqq \min \{\dim \mathbb{V}_j:  j \in [k]\}$.  Without loss of generality,  we may assume $\mathbb{V}_1 = \cdots = \mathbb{V}_k = \mathsf{K}^n$.  
\begin{itemize}
\item[$\diamond$] We first suppose that $|\mathsf{K}| = \infty$.  According to \cite[Theorem~3.7]{pielasa2025exact},  $\Q(S)= \lfloor\left(kn-k+1\right)^{1/(k-1)}\rfloor$ for any $S \in U_1$,  where $U_1 \subseteq (\mathsf{K}^{n})^{\otimes k}$ is nonempty and Zariski open.  By \cite[Lemma~4.1]{chen2025bounds} or \cite[Lemma~5.3]{kopparty2020geometric},  the geometric rank is lower semi-continuous.  Thus,  a generic $S \in (\overline{\mathsf{K}}^{n})^{\otimes k}$ has geometric rank $n$.  Since $\mathsf{K}$ is infinite,   $(\mathsf{K}^{n})^{\otimes k}$ is dense in $(\overline{\mathsf{K}}^{n})^{\otimes k}$ by Lemma~\ref{lem:density}.  This implies the existence of a nonempty Zariski open subset $U_2 \subseteq (\mathsf{K}^{n})^{\otimes k}$,  such that $\GR(S) = n$ for any $S \in U_2$.  Now,  for any $S_0 \in U_1 \cap U_2 \ne \emptyset$,  we have 
\[
\GR(S_0) = n,\quad Q(S_0) = \lfloor\left(kn-k+1\right)^{\frac{1}{k-1}}\rfloor. 
\]
\item[$\diamond$] Next,  we deal with the case where $|\mathsf{K}| = q$.  The proof is split into three steps: 
\begin{enumerate}[label = \arabic*):]
\item We claim that for any integer $r$ such that $n \ge r$ and $r^k \ge  nrk + 1$,  there are at least $(1 - q^{-1})q^{n^k}$ tensors in $(\mathsf{F}^n_q)^{\otimes k}$ with subrank at most $r$.  To prove the claim,  we consider the map 
\[
\psi_r: \GL_n(\mathsf{F}_q)^k \times D_r \to X_r,\quad \psi (g_1,\dots,  g_k,  S) = (g_1,\dots,  g_k) \cdot S
\]
where $(g_1,\dots,  g_k) \cdot S$ is defined as in \eqref{eq:action},  $X_r \coloneqq \{ S \in (\mathsf{F}^n_q)^{\otimes k}: \Q(T) \ge r\}$ and $D_r$ consists of all tensors $S \in (\mathsf{F}^n_q)^{\otimes k}$ such that $S_{j_1,\dots,  j_k} = $ if $(j_1, \cdots,  j_k) \in [r]^k\setminus \{(i,\dots,  i): i\in [r]\}$,  whereas $S_{i,\dots,  i} \ne 0$ for any $i \in [r]$.  It is straightforward to show that $\im(\psi_r) = X_r$ (cf.  \cite[Lemma~2.2]{derksen2024subrank}).  This implies that $|X_r| \le q^{n^k - r^k + n r k}$ and thus the claim.  Indeed,  we observe that for each $T \coloneqq (g_1,\dots,  g_k) \cdot S \in X_r$,  we have 
\[
\left\lbrace
(h_1,\dots, h_k,  (h^{-1}_1,\dots,  h^{-1}_k) \cdot T): h_1,\dots,  h_k \in L
\right\rbrace
 \subseteq \psi^{-1}( T ).  
\]
Here 
\[
L \coloneqq \left\lbrace \begin{bmatrix}
D & 0 \\
A & B
\end{bmatrix}: D \in  \GL_r(\mathsf{F}_q) \text{~diagonal},\; (A,B)\in \mathsf{F}_q^{(n-r) \times r} \times \GL_{n-r}(\mathsf{F}_q)
\right\rbrace.
\]
Since $|L| = (q-1)^r q^{r (n- r)} \smallprod_{s = 1}^{n - r - 1} (q^{n-r} - q^s)$ and $\vert \GL_{n}(\mathsf{F}_q)^k \times D_{r}\vert=q^{n^{k}-r^{k}} (q-1)^r\smallprod_{s=1}^{n-1} (q^{n}-q^s)^{k} $,  we have 
\[
|X_r|  \le \frac{\lvert  \GL_n(\mathsf{F}_q)^k \times D_r \rvert}{|L|^k}  \le \frac{q^{n^k - r^k + n r k}}{(q-1)^{(k-1)r}} \le q^{n^k - r^k + n r k}  \le q^{n^k - 1}.
\]
This implies that there are at least $(1 - q^{-1})q^{n^k}$ tensors in $(\mathsf{F}_q^n)^{\otimes k}$ with subrank at most $r$. 
\item We show that there are at least $ (1 - 2 q^{-n/2}) q^{n^k}$ tensors in $(\mathsf{F}_q^n)^{\otimes k}$ with analytic rank at least $n/2-1$.  To this end,  we regard tensors in $(\mathsf{F}_q^n)^{\otimes k}$ as $(k-1)$-multilinear maps \footnote{Strictly speaking,  $F_i$'s are $(k-1)$-multilinear functions on $( (\mathsf{F}_q^n)^\ast)^{k-1}$,  but we may ignore the duality by choosing a vector space isomorphism $(\mathsf{F}_q^n)^\ast \simeq \mathsf{F}_q^n$.} $F = (F_1,\dots,  F_n): \mathbb{V}_1 \times \cdots \times \mathbb{V}_{k-1} \to \mathsf{F}_q^n$,  where $\mathbb{V}_1 = \cdots = \mathbb{V}_{n-1} = (\mathsf{F}_q^n)^{k-1}$.  We consider a bipartite graph $(\mathcal{X}, \mathcal{Y},  \mathcal{E})$ with parts $\mathcal{X} \coloneqq \mathbb{V}_1 \times \cdots \times \mathbb{V}_{k-1}$ and $\mathcal{Y} \coloneqq \Hom(\mathbb{V}_1 \times \cdots \times \mathbb{V}_{k-1},  \mathsf{F}_q^n)$,  and edge set $\mathcal{E} \coloneqq \lbrace (v,F) \in \mathcal{X} \times \mathcal{Y}: F(v) = 0 \rbrace$.  We observe that for each $v = (v_1,\dots,  v_{k-1}) \in \mathcal{X}$ with $v_j = 0$ for some $1 \le j \le k-1$,  $(v,F) \in \mathcal{E}$ for all $F \in \mathcal{Y}$.  For $v = (v_1,\dots,  v_{k-1}) \in \mathcal{X}$,  it is obvious that $| \mathcal{E}_v | = q^{n^k}$ if some component of $v$ is zero,  where $\mathcal{E}_v \coloneqq \left\lbrace
F \in \mathcal{Y}:  (v,F) \in \mathcal{E}
\right\rbrace$.  If components of $v$ are all nonzero,  then $v$ imposes one linear equation on each component of $F = (F_1,\dots,  F_n) \in \mathcal{Y}$.  Thus,  we have $| \mathcal{E}_v | = q^{n^k - n}$.  Since there are $\binom{k-1}{s}(q^n - 1)^{k-1-s}$ elements in $\mathcal{X}$ with exactly $s$ zero elements,  we obtain 
\begin{align*}
\lvert \mathcal{E}  \rvert &\le q^{n^k + (k-2)n} + (q^{n^k} - q^{n^k-n}) \sum_{s=1}^{k-1} \binom{k-1}{s}(q^n - 1)^{k-1-s} \\
&\le q^{n^k + (k-2)n} + (q^{n^k} - q^{n^k-n})q^{n(k-2)} \sum_{s=1}^{k-1} \binom{k-1}{s} \\
&\le q^{n^k + (k-2)n} + (q^{n^k} - q^{n^k-n})q^{n(k-2)} q^{k-1} \\
&= q^{n^k + (k-2)n} + q^{n^k + (k-2)n + k-1} - q^{n^k + (k-3)n + k-1} \\
& = q^{n^k + (k-2)n + k-1} \left(  1 + \frac{1}{q^{k-1}} - \frac{1}{q^n} \right) \\
&\le 2q^{n^k + (k-2)n + k-1}
\end{align*}
Suppose that $\mathcal{Y}$ has $t$ vertices whose degrees are at most $q^{s}$ where $s = \lceil (k-3/2)n \rceil$.  Then $q^s (q^{n^k} - t)/2 \le \lvert \mathcal{E} \rvert \le 2q^{n^k + (k-2)n + k-1}$,  from which we obtain 
\[
t  \ge \left(  1 - 4q^{ (k-2)n +  k  - s -1} \right) q^{n^k} \ge \left( 1 - q^{-\frac{n}{2} -1 + k} \right) q^{n^k}.
\] 
By construction,  there are at least $\left( 1 - q^{-n/2 -1 + k} \right) q^{n^k}$ multilinear maps $F \in \Hom(\mathbb{V}_1 \times \cdots \times \mathbb{V}_{k-1},  \mathsf{F}_q^n)$ vanishing on at most $q^s$ points in $\mathbb{V}_1 \times \cdots \times \mathbb{V}_{k-1}$.  In particular,  we have $n \ge \AR(F) \ge (k-1)n - s \ge n/2 - 1$.  According to Lemma~\ref{grvsar},  we obtain $\GR(F) \asymp_k n$.
\item Suppose $n \ge 2k + 2$.  For any integer $r$ such that $r \le n $ and $n r k + 1 \le r^k$,  we may find some $T \in (\mathsf{F}_q^n)^{\otimes k}$ such that $\Q(T) \le r $ and  $\GR(T) \asymp_k n$.  In particular,  we may pick $r = \lceil (2kn)^{1/(k-1)} \rceil$ to obtain a tensor $T_0$ such that $\GR(T_0) \asymp_k \Q(T_0)^{k-1}$.  \qedhere
\end{enumerate}
\end{itemize}
\end{proof}
Theorem~\ref{cor:GvsS} has several straightforward  consequences,  each extending a known result for order-three tensors to arbitrary order.  We recall from Lemmas~\ref{prvsgr} and \ref{prvsar} that 
\[
\PR(T) \lesssim_k 
\begin{cases}
\GR(T)^{k-1} \quad &\text{if $|\mathsf{K}| = \infty$},  \\
\AR(T) \log_{|\mathsf{K}|} \AR(T) &\text{otherwise}. 
\end{cases}
\]
Hence Proposition~\ref{main theorem 2} and Theorem~\ref{cor:GvsS} yield the following relation between subrank and partition rank/analytic rank,  previously established for $k = 3$ in \cite[Corollaries~1.11 and 1.12]{chen2025bounds}.
\begin{corollary}[Subrank v.s partition rank and analytic rank]\label{cor:svsp}
Let $\mathsf{K}$ be a field and let $\mathbb{V}_1,\dots,  \mathsf{V}_k$ be vector spaces over $\mathsf{K}$.  For any $T \in \mathbb{V}_1 \otimes \cdots \otimes \mathbb{V}_k$,  we have 
\begin{enumerate}[label = (\alph*)]
\item $\Q(T) \le \PR(T) \lesssim_k \Q(T)^{(k-1)^{2}}$. 
\item If $\mathsf{K}$ is a finite field,  then $\Q(T) \le \AR(T) \lesssim_k \Q(T)^{k-1}$.
\end{enumerate}
\end{corollary}

Suppose $\mathsf{K}$ is a field and $\mathbb{V}_1,\dots,  \mathbb{V}_k$ are vector spaces over $\mathsf{K}$.  Each $\mathsf{K}$-tensor $T \in \mathbb{V}_1 \otimes \cdots \otimes \mathbb{V}_k$ extends to  a $\overline{\mathsf{K}}$-tensor $T^{\overline{\mathsf{K}}}\in (\mathbb{V}_1\otimes_{\mathsf{K}} \overline{\mathsf{K}}) \otimes_{\overline{\mathsf{K}}} \cdots \otimes_{\overline{\mathsf{K}}}  (\mathbb{V}_k \otimes_{\mathsf{K}} \overline{\mathsf{K}})$.  We denote
\[
\GR_{\overline{\mathsf{K}}} (T) \coloneqq \GR_{\overline{\mathsf{K}}} (T^{\overline{\mathsf{K}}}),\quad \Q_{\overline{\mathsf{K}}} (T) \coloneqq \Q_{\overline{\mathsf{K}}} (T^{\overline{\mathsf{K}}}),\quad \PR_{\overline{K}}(T) \coloneqq \PR_{\overline{\mathsf{K}}}( T^{\overline{\mathsf{K}}}).
\]
We notice that by definition $\GR_{\overline{\mathsf{K}}} (T) = \GR(T)$.  A combination of Theorem~\ref{cor:GvsS} and \cite[Theorem~5.1]{kopparty2020geometric} implies 
\[
\Q_{\overline{\mathsf{K}}}(T) \le \GR_{\overline{\mathsf{K}}}(T) = \GR(T)  \lesssim_k \Q (T)^{k-1} 
\]
from which we obtain the following generalization of \cite[Theorem~1.5]{biaggi2025real} and  \cite[Corollary~1.6]{chen2025bounds}.  This also resolves the conjecture on the stability of the subrank \cite[Conjecture~8.2]{chen2025bounds}.  
\begin{corollary}[Stability of subrank]\label{cor:stab}
Suppose $\mathsf{K}$ is a field and $\mathbb{V}_1,\dots,  \mathbb{V}_k$ are vector spaces over $\mathsf{K}$.  For any $T \in \mathbb{V}_1 \otimes \cdots \otimes \mathbb{V}_k$,  we have $\Q(T) \le \Q_{\overline{\mathsf{K}}}(T) \lesssim_k \Q(T)^{k-1}$. 
\end{corollary}

If $\mathsf{K}$ is an algebraically closed field,  then the \emph{border subrank} of $T \in \mathbb{V}_1 \otimes \cdots \otimes \mathbb{V}_k$ is defined as 
\[
\underline{\Q}(T) \coloneqq \max \left\lbrace
r: I_r \in \overline{ \left( \GL(\mathbb{V}_1) \times \cdots \GL(\mathbb{V}_k) \right) \cdot T}
\right\rbrace. 
\] 
By \cite[Theorem~5.1]{kopparty2020geometric},  we have $\underline{\Q}(T) \le \GR(T)$.  Combining this with Theorem~\ref{cor:GvsS},  we obtain the following de-bordering result for subrank,  generalizing \cite[Corollary~1.10]{chen2025bounds} from order three to arbitrary order.
\begin{corollary}[De-bordering of subrank]\label{cor:debor}
Suppose $\mathbb{V}_1,\dots,  \mathbb{V}_k$ are vector spaces over an algebraically closed field $\mathsf{K}$.  For any $T \in \mathbb{V}_1 \otimes \cdots \otimes \mathbb{V}_k$,  we have $\Q(T) \le \underline{\Q}(T) \lesssim_k \Q(T)^{k-1}$. 
\end{corollary}
\subsection{Ramsey problems}\label{subsec:Ramsey}
Given $F \in \Hom(\mathbb{V}_1 \times \cdots \times \mathbb{V}_d,  \mathbb{W})$,  recall that the kernel $\Ker(F)$ of $F$ defined in \eqref{eq:kerF} is a linear subspace of $\mathbb{V}_1 \otimes \cdots \otimes \mathbb{V}_d$.  In what follows,  we establish a Ramsey-type result on the extremal structures of $\Ker(F)$ described by the isotropy and completeness indices.  In \cite{qiao2020tur},  it is proved that an alternating bilinear map on a sufficiently large vector space must have either a large isotropy index,  or a large completeness index; the multilinear maps considered in Proposition~\ref{prop:ex} also exhibit the same extremal dichotomy.  Motivated by these examples,  we introduce the Ramsey numbers for multilinear maps.
\begin{definition}[Ramsey number]\label{def:Ramsey}
Given a field $\mathsf{K}$ and nonnegative integers $d,  s,  t$,  the Ramsey number $R(\mathsf{K},d,s,t)$ for multilinear maps is the minimal integer $m$ such that for any $\mathsf{K}$-vector spaces $\mathbb{V}_1,\dots,  \mathbb{V}_d,  \mathbb{W}$ with $\min_{j \in [d]} \dim \mathbb{V}_j  \ge m$ and $F \in \Hom(\mathbb{V}_1 \times \cdots \times \mathbb{V}_d,  \mathbb{W})$,  we either have $\alpha(F) \ge s$ or $\beta(F) \ge t$.  Similarly,  the Ramsey number $R_{\w}(\mathsf{K},d,s,t)$ (resp.  $R_{\s}(\mathsf{K}, d,s,t)$) for alternating (resp. symmetric) maps is the minimal integer $n$ such that for any $\mathsf{K}$-vector spaces $\mathbb{V},  \mathbb{W}$ with $\dim \mathbb{V} \ge n$ and $F \in \Alt^d(\mathbb{V},  \mathbb{W})$ (resp.  $F \in \Sym^d(\mathbb{V},  \mathbb{W})$),  we either have $\alpha_{\w}(F) \ge s$ or $\beta_{\w}(F) \ge t$ (resp.  $\alpha_{\s}(F) \ge s$ or $\beta_{\s}(F) \ge t$).
\end{definition} 
Below we establish the existence and upper bounds of the Ramsey numbers,  which are the content of Theorem~\ref{prop:Ramsey}.
\begin{proof}[Proof of Theorem~\ref{prop:Ramsey}]
The proofs of \ref{prop:Ramsey:item1}–\ref{prop:Ramsey:item3} are identical,  except for invoking the corresponding items (a)--(c) of Theorems~\ref{lem-isotropic-general} and \ref{lem-complete1},  respectively.  Thus,  we only present the proof of \ref{prop:Ramsey:item1} below. According to Lemmas~\ref{prvsgr} and \ref{prvsar},  there is a function $c: \mathbb{N} \to \mathbb{N}$ such that 
\begin{equation}\label{prop:Ramsey:eq1}
\PR(F) \le 
\begin{cases}
c(d) \GR(F)^d \quad &\text{if~}|\mathsf{K}| = \infty \\
c(d) \AR(F) \log_q \AR(F) \quad &\text{if~} |\mathsf{K}| = q
\end{cases}.
\end{equation}
In the following,  we denote $c \coloneqq c(d)$.  For infinite $\mathsf{K}$,  we let $n \coloneqq 2 c (t+1)^{d^2} s^{d-1}$.  If $\GR(F) \le (n / (2 c s^{d-1}))^{1/d}$,  then we have $\PR(F) \le n /(2s^{d-1})$.  By Theorem~\ref{lem-isotropic-general}--\ref{lem-isotropic-general:item1},  we may deduce that $\alpha(F)\ge s$.  If $\GR(F) \ge (n / (2 c s^{d-1}))^{1/d}$,  then Theorem~\ref{lem-complete1}--\ref{lem-complete1:gen} implies 
\[
\beta \ge \GR(F)^{1/d} \ge (n / (2 c s^{d-1}))^{1/d^2} = t + 1.
\]
If $|\mathsf{K}| = q$,  we let $n \coloneqq 4c f(t) s^{d-1}\log_q(2c f(t) s^{d-1})$,  where $f(t) \coloneqq d + (t+1)^d + \lceil (t+1)^d (d-1)/(q-1) \rceil$.  If $c \AR(F) \log_q \AR(T) \le n/(2s^{d-1})$, then we have $\PR(F) \le n/ (2s^{d-1})$ and Theorem~\ref{lem-isotropic-general}--\ref{lem-isotropic-general:item1} implies $\alpha(F)\ge s$.  If $c \AR(F) \log_q(\AR(T)) \ge n/(2s^{d-1})$,  then we have 
\[
\AR(F) \ge \frac{n}{2c s^{d-1} \log_q \AR(F)} \ge \frac{n}{2c s^{d-1} \log_q n} \ge f(t).
\]
According to Theorem~\ref{lem-complete1}--\ref{lem-complete1:gen},  we obtain $\beta(F) \ge t$. 
\end{proof}
\begin{remark}
Note that in Theorem~\ref{prop:Ramsey},  we require $\mathsf{K} = \overline{\mathsf{K}}$ for the existence of $R_{\s}(\mathsf{K},  d,s,t)$.  Indeed,  the Ramsey number for symmetric multilinear maps may not exist over a non-algebraically closed field.  For example,  take a positive integer $n$ and consider the standard inner product $F: \mathbb{R}^n \times \mathbb{R}^n \to \mathbb{R}$ on $\mathbb{R}^n$.  It is clear that $F \in \Sym^2(\mathbb{R}^n,  \mathbb{R})$,  $\alpha_{\s} (F) = 0$ and $\beta_{\s}(F) = 1$.  In contrast,  we have $\alpha (F) = \lfloor n/2 \rfloor$ and $\beta(F) = 1$,  which is compatible with Theorem~\ref{prop:Ramsey}--\ref{prop:Ramsey:item1}.
\end{remark}

We notice that the upper bounds for the Ramsey numbers in Theorem~\ref{prop:Ramsey} rely on the upper bound \eqref{prop:Ramsey:eq1} of the partition rank in terms of the geometric rank or analytic rank.  Therefore,  any strengthening of~\eqref{prop:Ramsey:eq1} would immediately yield tighter bounds in Theorem~\ref{prop:Ramsey}. 
\begin{corollary}[Improved upper bounds for Ramsey numbers]\label{special bound}
Let $\mathsf{K}$ be a field.  If $\PR(T) \lesssim_k \GR(T)$ for any positive integer $k$ and any $\mathsf{K}$-tensor $T$ of order $k$,  then $R(\mathsf{K},d,s,t,) \lesssim_d s^{d-1} t^d$ for any nonnegative integers $d,s$ and $t$.  In particular,  we have the following: 
\begin{enumerate}[label = (\alph*)]
        \item\label{special bound:item1} If $\mathsf{K}$ is algebraically closed,  then $R(\mathsf{K}, d, s,t) \lesssim_d s^{d-1}t^{d}$.
        \item\label{special bound:item2} If $\mathsf{K}$ is perfect,  then $R(\mathsf{K}, 2, s,t) \lesssim st^{2}$.
      \item\label{special bound:item3} If $\mathsf{K}$ is infinite with transcendence degree $n$ over its prime field,  then $R(\mathsf{K},d,s,t) \lesssim_{d,n} s^{d-1}t^{d}$.
    \end{enumerate}
Moreover,  these bounds remain valid for $R_{\w}(\mathsf{K},d,s,t)$.
\end{corollary}
\begin{proof}
Replacing the inequality in \eqref{prop:Ramsey:eq1} by $\PR(F) \le  c(d) \GR(F)$ and repeating the rest of the proof of Theorem~\ref{prop:Ramsey},  we obtain $R(\mathsf{K},d,s,t,) \lesssim_d s^{d-1} t^d$.  Special cases \ref{special bound:item1}--\ref{special bound:item3} follow immediately from Lemma~\ref{prvsgr}.
\end{proof}
Next,  we prove Proposition~\ref{ramesey lower bound}.
\begin{proof}[Proof of Proposition~\ref{ramesey lower bound}]
Denote $\alpha \coloneqq s-1$, $m  \coloneqq \binom{t}{d}-1$ and $n  \coloneqq  \binom{\alpha}{d} m /\alpha + \alpha$.  By \cite[Section~5]{Dav1992number}, there exists an $F \in \Alt^d(\mathsf{K}^n,  \mathsf{K}^m)$ such that $\alpha_{\w} (F) \le \alpha$.  By Lemma~\ref{lem:basic},  we also have $\beta_{\w}(F) \le t - 1$, and this implies $R_{\w}(\mathsf{K},  d,  s,t)\ge n + 1$.  The ``in particular" part follows immediately from Corollary~\ref{special bound}.
\end{proof}

The remainder of this subsection discusses the implications of Theorem~\ref{prop:Ramsey} for multilinear algebra,  group theory and algebraic geometry.  Let $\mathbb{V} \subseteq \mathsf{K}^{n_1} \otimes \cdots  \otimes \mathsf{K}^{n_d}$ be a linear subspace.  We say that $\mathbb{V}$ \emph{has CP-rank at most $r$} if $\cp-rank(T) \le r$ for each $T \in \mathbb{V}$,  i.e.,  $T = \sum_{j=1}^r v_{j,1} \otimes \cdots \otimes v_{j,d}$ for some $v_{j,i} \in \mathsf{K}^{n_i}$,  $(i,j) \in [d] \times [r]$,  and that $\mathbb{V}$ is an \emph{$(l_1,  \dots,  l_d)$-decomposable space} if every $T = (T_{k_1,\dots,  k_d}) \in \mathbb{V}$ satisfies $T_{k_1,\dots,  k_d} = 0$ whenever $k_1 > l_1,\dots,  k_d > l_d$.  Two linear subspaces $\mathbb{V},  \mathbb{W} \subseteq \mathsf{K}^{n_1} \otimes \cdots  \otimes \mathsf{K}^{n_d}$ are \emph{equivalent} if $\mathbb{W} = (g_1,  \cdots,  g_d) \cdot \mathbb{V}$ for some $(g_1,\dots,  g_d) \in \GL_{n_1}(\mathsf{K}) \times \cdots \times  \GL_{n_d}(\mathsf{K})$.  Here the operation $\cdot$ is defined in \eqref{eq:action}.  For $d = 2$,  above definitions were given in \cite{AL80} to study linear spaces of matrices of bounded rank.  We observe that $\mathbb{V}$ uniquely determines a $d$-multilinear map $F_{\mathbb{V}} \in \Hom(\mathsf{K}^{n_1} \times \cdots \times  \mathsf{K}^{ n_d},  \mathsf{K}^{\dim \mathbb{V}})$.  Consequently,  $\mathbb{V}$ is $(n_1-\alpha(F_{\mathbb{V}}),\dots,  n_d-\alpha(F_{\mathbb{V}}))$-decomposable and $\mathbb{V}$ must contain a tensor of CP-rank at least $\lceil \beta(F_{\mathbb{V}}))^d/(d  \beta(F_{\mathbb{V}}) - d + 1) \rceil$.  The first assertion is clear from the definition; the second follows from the fact \cite[Proposition~14.29]{BCS97} that $\cp-rank(S) \le \cp-rank(T)$ for any tensors $S,  T$ satisfying $S  \unlhd T$.  Viewing $\mathbb{V}$ as $F_{\mathbb{V}} \in \Hom(\mathsf{K}^{n_1} \times \cdots \times  \mathsf{K}^{ n_d},  \mathsf{K}^{\dim \mathbb{V}})$ and applying Theorem~\ref{prop:Ramsey} and Corollary~\ref{special bound},  we obtain the corollary that follows.
\begin{corollary}[Ramsey-type result for linear spaces of tensors]\label{cor:matrixspace}
Let $\mathsf{K}$ be a field and let $d,  s,  t$ be positive integers. For any integers $n_1,  \dots,  n_d \ge R(\mathsf{K},d,s,t)$,  every linear subspace of $\mathsf{K}^{n_1} \otimes \cdots  \otimes \mathsf{K}^{n_d}$ is either equivalent to an $(n_1-s, \dots,  n_d - s)$-decomposable space,  or it contains a tensor of CP-rank at least $\lceil t^d/(d  t - d + 1) \rceil$.  The same holds for alternating and symmetric matrices, respectively; in the symmetric case,  we assume $\overline{\mathsf{K}} = \mathsf{K}$.
\end{corollary}
\begin{remark}
Corollary~\ref{cor:matrixspace} remains valid with the CP-rank replaced by other tensor ranks.  Taking $d = 2$,  $ \mathbb{V}$ becomes a linear space of matrices,  which has been extensively studied for over sixty years.  In this case,  the extremal structures appearing in Corollary~\ref{cor:matrixspace} have been observed when $\dim \mathbb{V}$ is large.  Given a linear subspace $\mathbb{V} \subseteq \mathsf{K}^{n_1 \times n_2}$,  \cite{AL80,Pazzis15} proved that if $\dim \mathbb{V} > r \max\{n_1-1,n_2-1\}$ and $\mathbb{V}$ has rank at most $r$,  then $\mathbb{V} $ is decomposable.  On the other hand,  it is well-known \cite{Flanders62, Meshulam85} that $\mathbb{V}$ must contain a matrix of rank at least $(r+1)$ whenever $\dim \mathbb{V} > r \max\{n_1,n_2\}$.  Moreover,  \cite[Lemma~1]{Flanders62} implies that it suffices to assume that $n_1,  n_2 \ge s + t$ if $|\mathsf{K}| \ge s$ in Corollary~\ref{cor:matrixspace}.  Other structure theorems abound: for instance,  \cite[Theorem~1]{AL81} shows that every linear space of bounded-rank matrices is equivalent to the direct sum of a decomposable and a primitive space; this yields complete classifications up to rank four \cite{Atkinson83,EH88,HL23}.  Notice that in all these results the matrix size is fixed, whereas in Corollary~\ref{cor:matrixspace} it varies with $s$ and $t$.  
\end{remark}

We recall that Proposition~\ref{ramesey lower bound} optimally improves the existing upper bound of $R_{\w}(\mathsf{K},2,s,t)$ \cite[Theorem~4.1]{qiao2020tur} when $\mathsf{K}$ is perfect.  As a result,  we obtain the following improved Ramsey-type result for finite $p$-groups,  which follows from the same argument as in \cite[Corollary~4.2]{qiao2020tur}.
\begin{corollary}[Ramsey-type result for $p$-groups]\label{cor:Ramseygroup}
Let $p$ be a prime number and let $s,t$ be positive integers.  Suppose that $G$ is a $p$-group $G$ satisfying $[G,G] \subseteq Z(G)$.  If $G$ is generated by at least $R_{\w}(\mathsf{K},2,s,t) \asymp st^2$ elements,  then $G$ has a subgroup $H$ such that either $H$ is abelian and $H / \left( H \cap [G,G] \right) \simeq \mathsf{F}_p^s$,  or $H$ is generated by $t$ elements and $[H,[H,H]] = \{1_G \}$.  Here $Z(G)$ denotes the center of $G$ and $1_G$ is the identity element in $G$.
\end{corollary}
In a similar vein,  we may generalize and strengthen \cite[Corollary~3]{Dav1992number} and \cite[Corollary~4.4]{qiao2020tur}.
\begin{corollary}[Ramsey-type result for linear sections of Grassmannians]\label{cor:RamseyGr}
Let $\mathsf{K}$ be a field and let $d,s,t$ be positive integers. For any integer $n \ge R_{\w}(\mathsf{K},d,s,t)$ and any linear section $X$ of $\Gr(d,\mathsf{K}^n)$ in $\mathbb{P} (\Lambda^d \mathsf{K}^n)$,  we either have $\Gr(d,\mathbb{U}) \subseteq X$ for some $\mathbb{U} \in \Gr(s,  \mathsf{K}^n)$,  or $X \cap \Gr(d,\mathbb{W}) = \emptyset$ for some $\mathbb{W} \in \Gr(t,  \mathsf{K}^n)$.  Here $\Gr(k,\mathbb{V})$ is the Grassman variety of $k$-dimensional subspaces of $\mathbb{V}$.
\end{corollary}  

Suppose $\mathsf{K}$ is algebraically closed.  Given $F \in \Sym^d(\mathsf{K}^{N+1},  \mathsf{K}^n)$,  Theorem~\ref{lem-complete1}--\ref{lem-complete1:sym} shows that $\beta_{\s}(F) \le \height(\mathfrak{a}_F)$.  This together with Theorem~\ref{prop:Ramsey} yields the following Ramsey-type result.  
\begin{corollary}[Ramsey-type result for projective varieties]\label{cor:var}
Let $\mathsf{K}$ be an algebraically closed field and let $d,s,t$ be positive integers.  Suppose $X$ is a subvariety of $\mathbb{P}^N$ defined by homogeneous polynomials of degree $d$ over $\mathsf{K}$.  If $N \ge R_{\s}(\mathsf{K},d,s+1,t)  \asymp_d s^{d-1} t^d$,  then $X$ either contains an $s$-dimensional subspace,  or $\codim X \ge t$.  In particular,  if $N \ge R_{\s}(\mathsf{K},d,s+1,2) \asymp_d s^{d-1}$,  then every degree-$d$ hypersurface in $\mathbb{P}^N$ must contain an $s$-dimensional linear subspace.   
\end{corollary}
\begin{remark}
Corollary~\ref{cor:var} asserts that varieties of small codimension must contain a large linear subspace,  consistent with the guiding principle \cite{Landsberg94,Hartshorne74} that smooth subvarieties of small codimension should behave like hypersurfaces.  

Given a variety $X \subseteq \mathbb{P}^N$ and a positive integer $k$,  we denote by $F_k(X) \subseteq \Gr(k+1,\mathsf{K}^{N+1})$ its $k$-th Fano variety,  consisting of all $k$-dimensional linear subspaces contained in $X$.  Fano varieties of smooth hypersurfaces have been long investigated by algebraic and differential geometers \cite{Landsberg94,HMP98,Beheshti05,Beheshti06,LR10,Debarre17,BR21}.  By \cite[Theorem~1.3]{BR21},  if $N \ge 2\binom{s + d  - 1}{d-1} + s-1 \asymp_d s^{d-1}$ and $X \subseteq \mathbb{P}^n$ is a smooth hypersurface,  then $F_s(X)$ is of the expected dimension $(s+1)(N-s) - \binom{s+d}{d} > 0$.  In particular,  this verifies Corollary~\ref{cor:var} for smooth hypersurfaces.  The fact that Fano varieties are of expected dimension for a generic variety \cite[Theorem~2.1]{DM98} and some low-degree smooth complete intersections \cite[Theorem~1.3]{Canning21} further confirms Corollary~\ref{cor:var}.  However,  the dichotomous structures observed in Corollary~\ref{cor:var},  as far as we are aware,  have not been discussed in full generality. 
\end{remark}

\subsection{Probabilistic algorithms}The primary goal of this subsection is to show that the completeness index can be computed in polynomial time by a probabilistic algorithm .  We also briefly discuss its immediate applications to computation of tensor ranks and height of ideals.  The following proposition is essentially due to Youming Qiao,  who asserted its validity for $\beta_{\w}$ with $d = 2$ in \cite[Section 5.8]{qiao2020tur} and kindly explained the idea to us in a private communication. 
\begin{proposition}[Probabilistic algorithm for complete index]\label{random compute beta}
For any $\varepsilon>0$,  there is a probabilistic algorithm  with error probability at most $\varepsilon$ that computes $\beta(F)$ (resp.  $\beta_{\w}(F)$,  $\beta_{\s}(F)$) for $F \in \Hom(\mathsf{F}_q^{n_1} \times \cdots \times \mathsf{F}_q^{n_d},  \mathsf{F}_q^n)$ (resp.  $F \in \Alt^d(\mathsf{F}_q^{n_1},  \mathsf{F}_q^n)$,  $F \in \Sym^d(\mathsf{F}_q^{n_1},  \mathsf{F}_q^n)$) within $\poly(d,N,-\log \varepsilon)$ field operations,  provided $q > 2dN$.  Here $N \coloneqq n_{1}+\cdots+n_d+n$ (resp.  $N \coloneqq d n_1 + n$,  $N \coloneqq d n_1 + n$).
\end{proposition}
\begin{proof}
We confine ourselves to the probabilistic algorithm for general multilinear maps,  since the algorithms for alternating and symmetric cases are similar.  We denote $m \coloneqq \min \{n_1,\dots,  n_d,  n^{1/d}\}$.  By definition,  we have $\beta(F)\le m$.  We claim that for each $c\in [m]$,  there is a probabilistic algorithm that decides whether $\beta(F) \le c$ with error probability at most $\varepsilon/m$,  using $\poly(d,N,-\log \varepsilon )$ field operations.  Assuming the claim,  we can determine $\beta(F)$ with error probability at most $1-(1-\varepsilon/m)^{m}\le \varepsilon$,  using $\poly(d,N,-\log \varepsilon )$ field operations.

Thus,  it is left to prove the claim.  Let $\mathcal{L}_{c+1}$ be the set defined as in \eqref{lem-complete1:gen:eq-1}.  We choose an ordering on $\mathcal{L}_{c+1}$ so that its elements $\mathbf{L}_1,\dots,  \mathbf{L}_{L_{c+1}}$ are listed as in \eqref{lem-complete1:gen:eq0}.  Here $L_{c+1} \coloneqq (c+1)^d$.  Suppose $R$ is the polynomial ring over $\mathsf{F}_q$ with variables $w_{i,j}$ where $(i,j) \in [d] \times [L_{c+1}]$.  We consider the polynomial matrix $M = (M_{s,t})_{s,t=1}^{L_{c+1},n} \in R^{L_{c+1} \times n}$ with 
\[
M_{s,t} \coloneqq F_t(w_{1,\mathbf{L}_s(1)}, \dots,  w_{d,  \mathbf{L}_s(d)} ),\quad (s,t) \in [L_{c+1}] \times [n].  
\]
Note that $\beta (F) < c+1$ if and only $\rank M(v) < L_{c+1}$ for any $v \in \mathsf{F}_q^{n_1} \times \cdots \times \mathsf{F}_q^{n_d}$.  Since the maximal minors of $M$ are of degree at most $d L_{c+1} \le d N < q$,  Lemma~\ref{Schwartz-Zippel lemma} implies 
\[
\Pr \left( \rank M(v) < L_{c+1} \right) \le \frac{d L_{c+1}}{q} \le 1/2. 
\]
For independently and randomly sampled $k\coloneqq \lceil -2 \log(\varepsilon/m) \rceil$ points $v_1,\dots,  v_k \in \mathsf{F}_q^{n_1} \times \cdots \times \mathsf{F}_q^{n_d}$,  we have 
\begin{equation}\label{random compute beta:eq1}
\prod_{i=1}^k \Pr \left( \rank M(v_i) < L_{c+1} \right) \le 4^{\log \frac{\varepsilon}{m}} \le \frac{\varepsilon}{m}.
\end{equation}
The probabilistic algorithm is as follows.  If $\rank M(v_i) =L_{c+1}$ for some $i \in [k]$, it returns ``$\beta (T) \ge c+1$".  Otherwise,   it returns "$\beta(F)\le c$".  According to \eqref{random compute beta:eq1},  the error probability of this algorithm is at most $\varepsilon/m$.  Moreover,  we recall that evaluating $M$ at a point $v \in \mathsf{F}_q^{n_1} \times \cdots \times \mathsf{F}_q^{n_d}$ costs $O(N^d)$ field operations,  and checking the rank of $M(v) \in \mathsf{F}_q^{L_{c+1} \times n}$ costs $O(n^3)$ field operations.  Thus,  the probabilistic algorithm requires $O(k( N^{d}+n^{3}))= \poly(d,N,-\log \varepsilon )$ field operations.
\end{proof}
\begin{remark}\label{rmk:random compute beta}
If $|\mathsf{K}| = \infty$ and $\mathbb{V} \coloneqq \mathsf{K}^{n_1} \times \cdots \times \mathsf{K}^{n_d}$ is equipped with a probability distribution such that any subvariety $X \subsetneq \mathbb{V}$ has measure zero,  then one can decide whether $\beta(F) \le c$ using the algorithm in the proof of Proposition~\ref{random compute beta}.  In this case,  it is sufficient to check whether $\rank M(v) = L_{c + 1}$ for a single random point $v \in \mathbb{V}$.  The error probability is clearly zero and the cost of field operations is $O( (n_1+ \cdots + n_d)^d)$ if $d \ge 3$.
\end{remark}
For each $\varepsilon > 0$,  we combine Propositions~\ref{random compute beta} and \ref{main theorem 2} to obtain a probabilistic algorithm with error probability at most $\varepsilon$ that computes a number $h$ for $T \in \mathsf{F}_q^{n_1} \otimes \cdots \otimes \mathsf{F}_q^{n_k}$ such that 
\begin{equation}\label{eq:prob}
h \le \Q(T) \le \underline{\Q}(T) \le  \GR(T) \le \frac{\AR(T)}{1 - \log_q k} <  \frac{k-1 + (h+1)^{k-1} + \left\lceil \frac{(h+1)^{k-1}(k-2)}{q-1} \right\rceil}{1 - \log_q k},
\end{equation}
within $\poly(q,-\log \varepsilon )$ field operations,  provided $(k-1)N < q$.  Here $N \coloneqq n_1 + \cdots + n_k$ and the preultimate inequality follows from \cite[Proof of Theorem~1.13--(1)]{adiprasito2021schmidt}.  Since geometric rank is invariant under field extension,  we may remove the constraint $(k-1)N < q$ for it.
\begin{corollary}[Probabilistic algorithm for bounds of geometric rank]\label{cor:probgr}
For any $\varepsilon > 0$,  there is a probabilistic algorithm with error probability at most $\varepsilon$ that computes a number $h$ for any $T \in \mathsf{F}_q^{n_1} \otimes \cdots \otimes \mathsf{F}_q^{n_k}$ such that
\[
h \le \GR(
T)< (1 - \log_q k)^{-1} \left(k-1 + (h+1)^{k-1} + \left\lceil \frac{(h+1)^{k-1}(k-2)}{q-1} \right\rceil \right),
\]
within $\poly(N,k,-\log \varepsilon)$ field operations where $N \coloneqq n_1 + \cdots + n_k$.
\end{corollary}
\begin{proof}
Let $l$ be the smallest positive integer such that $q^l> (k-1) N$.  We recall from \cite[Chapters~2 and 3]{burgisser2013algebraic} that any arithmetic operation in $\mathsf{F}_{q^{l}}$ can be performed using  $\widetilde{O}(\log_2 ((k-1)N)$ operations in $\mathsf{F}_{q}$.  The existence of the algorithm follows immediately from \eqref{eq:prob} over $\mathsf{F}_{q^l}$ and the invariance of geometric rank under field extension.
\end{proof}
\begin{remark} 
According to \cite{Koiran97},  the problem of computing the geometric rank lies in the classes $\PSPACE$ and $\AM$ under certain assumptions; see also \cite[Section~2]{kopparty2020geometric}.  However,  the computational complexity-and indeed the hardness-of computing the geometric rank remains open.  Corollary~\ref{cor:probgr} is the first step toward this open problem,  although it falls far short of solving it.
\end{remark}

Since the height of a polynomial ideal is also invariant under field extension \cite[Lemma~3.4]{chen2025bounds},  the same argument used for Corollary~\ref{cor:probgr},  along with Theorem~\ref{lem-complete1},  Proposition~\ref{random compute beta} and Remark~\ref{rmk:random compute beta},  yields the following.
\begin{corollary}[Probabilistic algorithm for bounds of height]\label{cor:probht}
For any $\varepsilon > 0$,  there is a probabilistic algorithm with error probability at most $\varepsilon$ that computes a number $h$ for any homogeneous polynomials $F_1,\dots,  F_m \in  \mathsf{F}_q[x_1,\dots,  x_n]$ of degree $d$ such that $h \le \height (\mathfrak{a}) < \binom{h + d}{d}$ within $\poly(dn,-\log \varepsilon)$,  where $\mathfrak{a} = (F_1,\dots,  F_m)$.  Moreover,  if $\mathsf{K}$ is infinite and $\mathsf{K}^n$ is equipped with a probability distribution such that any subvariety $X \subsetneq \mathsf{K}^n$ has measure zero,  then there is a probabilistic algorithm that computes a number $h$ for any homogeneous polynomials $F_1,\dots,  F_m \in  \mathsf{K}[x_1,\dots,  x_n]$ of degree $d$ such that $h \le \height (\mathfrak{a}) < \binom{h + d}{d}$ within $O( (dn + m)^d)$ field operations.  
\end{corollary}
\begin{remark}
According to \cite[Proposition~1.1]{Koiran97},  the problem of deciding whether $\height \mathfrak{a} \le n-D$ is NP-hard for any $D \le n$.  Here $\mathfrak{a}$ is an ideal in $\mathbb{C}[x_1,\dots,  x_n]$ generated by polynomials of degree at most $d$.  Existing deterministic algorithms for height run in time polynomial in either $d^n$ \cite[Theorem~1.1]{Chistov96} or $D^n$ \cite{Koiran97}.   In contrast,  the probabilistic algorithm in Corollary~\ref{cor:probht} has complexity polynomial in $n^d$. 

It is worth mentioning that Corollary~\ref{cor:probht} provides a computable degree bound for Gr\"{o}bner bases.  By \cite[Theorem~3]{MR13},  the degree of the reduced Gr\"{o}bner basis $G$ of the ideal $\mathfrak{a} \subsetneq \mathsf{K}[x_1,\dots,  x_n]$ generated by homogeneous polynomials $F_1,\dots, F_m$ of degree $d$ is bounded by 
\begin{equation}\label{eq:Grob1}
\deg G \le 2 \left( \frac{d^{t} + d}{2} \right)^{2^{n- t - 1}},\quad t \coloneqq \height (\mathfrak{a}).
\end{equation}
This upper bound is not practical since computing $t$ is difficult,  as mentioned above.  However,  \eqref{eq:Grob1} together with Corollary~\ref{cor:probht} gives
\begin{equation}\label{eq:Grob2}
\deg G \le 2 \left( \frac{d^{\binom{h + d}{d}} + d}{2} \right)^{2^{n- h - 1}},
\end{equation}
where $h$ is a number probabilistically computable in time $O( (dn + m)^d )$.  We also notice that the bound in \eqref{eq:Grob2} is better than the one in \cite[Theorem~8.2]{Dube90} whenever $2^h > \binom{h+d}{d}$.
\end{remark}


\subsection*{Acknowledgment}
We thank Hong Liu and Zixiang Xu for introducing the work of Youming Qiao to us.  We also thank Youming Qiao for helpful discussions.
\bibliographystyle{abbrv}
\bibliography{ref}
\appendix
\end{document}